\documentclass{article}
\usepackage{latexsym,amsfonts,amsmath,amsthm}
\usepackage[margin=1in]{geometry}
\newcommand{\cqfd}{\mbox{}\nolinebreak\hfill\rule{2mm}{2mm}\medbreak\par}
\numberwithin{equation}{section}
\newtheorem{theorem}{Theorem}[section]

\newtheorem{lemma}{Lemma}[section]
\newtheorem{proposition}{Proposition}[section]
\newtheorem{remark}{Remark}[section]
\newtheorem{example}{Example}[section]
\newtheorem{corollary}{Corollary}[section]
\def\proof{\mbox {\it Proof.~}}
\title{Systems with discrete singular $\phi$-Laplacian and maximal monotone boundary conditions}

\author{Andreea GRUIE, Petru JEBELEAN and C\u{a}lin \c{S}ERBAN}
\newcommand{\Addresses}{{
  \bigskip
  \footnotesize

  A.~GRUIE, \textsc{Department of Mathematics, West University of Timi\c{s}oara, Blvd. V. P\^{a}rvan, no. 4, 300223  Timi\c{s}oara, Romania,}
  \textit{E-mail}: \texttt{andreea.gruie@e-uvt.ro}

  \medskip

  P.~JEBELEAN (Corresponding author), \textsc{Institute for Advanced Environmental Research, West University of Timi\c{s}oara, Blvd. V. P\^{a}rvan, no. 4, 300223  Timi\c{s}oara, Romania,}
  \textit{E-mail}: \texttt{petru.jebelean@e-uvt.ro}

  \medskip

  C.~\c{S}ERBAN, \textsc{Department of Mathematics, West University of Timi\c{s}oara, Blvd. V. P\^{a}rvan, no. 4, 300223  Timi\c{s}oara, Romania,}
  \textit{E-mail}: \texttt{calin.serban@e-uvt.ro}

}}

\date{}
\begin{document}
\maketitle

\begin{abstract}
We are concerned with solvability of nonlinear systems involving a discrete singular $\phi$-Laplacian operator of type
\begin{equation*}\label{sgen} u \mapsto \Delta\left[\phi(\Delta u(n-1))\right] \qquad (n\in \{1, \dots, T\}),
\end{equation*}
associated with a general two point boundary condition having the form
\begin{equation*}
\left(\phi(\Delta u(0)),-\phi(\Delta u(T))\right)\in\gamma(u(0),u(T+1)),
\end{equation*}
where $\gamma:\mathbb{R}^N\times\mathbb{R}^N\to2^{\mathbb{R}^N\times\mathbb{R}^N}$ is a maximal monotone operator with $0_{\mathbb{R}^N \times \mathbb{R}^N}\in \gamma(0_{\mathbb{R}^N \times \mathbb{R}^N})$.
The mapping $\phi$ is a potential homeomorphism from an open ball of radius $a$ centered at the origin $B_a \subset \mathbb{R}^N$  onto  $\mathbb{R}^N$ and $\Delta$ stands for the usual forward difference operator. When the perturbing nonlinearity in the system has not a potential structure we obtain existence of solutions by a priori estimates. Also, when the nonlinearity is of gradient type and $\gamma$ is a subdifferential, we provide a variational approach of the system in the frame of critical point theory for convex, lower semicontinuous perturbations of $C^1$-functionals. Then we derive the existence of solutions either as minimizers or saddle points of the corresponding energy functional.
\end{abstract}
\bigskip

\medskip

\bigskip

\section{Introduction}

In this paper we first deal with the unique solvability of the discrete system
\begin{equation}\label{syssimpl}
-\Delta\left[\phi(\Delta u(n-1))\right]+u(n)=h(n)  \quad (n \in \mathbb{Z}[1,T]),
\end{equation}
subject to a general two point boundary condition of type
\begin{equation}\label{bcsysgen}
\left(\phi(\Delta u(0)),-\phi(\Delta u(T))\right)\in\gamma(u(0),u(T+1)).
\end{equation}
Here and hereafter, for $p,q\in\mathbb{Z}$ with $p<q$, the notation $\mathbb{Z}[p,q]$ stands for the discrete interval $\{p, p+1,\ldots, q\}$ and $\Delta u(n-1):=u(n)-u(n-1)$ is the usual forward difference operator.
 We denote by $B_a$ the open ball in $\mathbb{R}^N$ -- endowed with the Euclidean norm $|\cdot|$, centered in $0_{\mathbb{R}^N}$ of radius $a>0$ and we assume the following hypotheses on the data entering in problem \eqref{syssimpl}-\eqref{bcsysgen}:
$$\begin{array}{cl}
(H_{\Phi})  & \phi \, \, \mbox{\sl is a homeomorphism from } B_a \, \, \mbox{\sl onto } \mathbb{R}^N, \, \, \mbox{\sl such that }\phi(0_{\mathbb{R}^N})=0_{\mathbb{R}^N}, \, \phi =\nabla \Phi, \, \, \mbox{\sl with }
\Phi: \overline{B}_a \to \mathbb{R} \ \mbox{\sl of} \\ & \mbox{\sl class } C^1  \mbox{\sl on } B_a, \, \mbox{\sl continuous, strictly convex on } \overline{B}_a \mbox{ and } \Phi(0_{\mathbb{R}^N})=0;\\
(H_{\gamma})  & \gamma:\mathbb{R}^N\times\mathbb{R}^N\to2^{\mathbb{R}^N\times\mathbb{R}^N} \mbox{\sl is maximal monotone with } 0_{\mathbb{R}^N \times \mathbb{R}^N}\in \gamma(0_{\mathbb{R}^N \times \mathbb{R}^N});\\
(H_{h})  & h:\mathbb{Z}[1,T] \to \mathbb{R}^N \mbox{\sl is a given discrete function}.
\end{array}$$
The inner product $\langle\cdot|\cdot\rangle$ engendering the norm $|\cdot|$ is considered on $\mathbb{R}^N$ and the product space $\mathbb{R}^N\times\mathbb{R}^N$ is endowed with its usual inner product, denoted by $\langle\!\langle\cdot|\cdot\rangle\!\rangle$. The monotonicity assumption in ($H_{\gamma}$) is understood in the sense of $\langle\!\langle\cdot|\cdot\rangle\!\rangle$ and the notation $\|\cdot\|$ will stand for the norm corresponding to $\langle\!\langle\cdot|\cdot\rangle\!\rangle$.

\medskip
By {\sl a solution} of \eqref{syssimpl}-\eqref{bcsysgen} we mean a function $u:\mathbb{Z}[0,T+1] \to \mathbb{R}^N$ with
$$\|\Delta u \|_{\infty}:=\displaystyle \max_{n\in \mathbb{Z}[1,T+1]} | \Delta u(n-1) |<a\, \mbox{ and } \, (u(0),u(T+1)) \in D(\gamma),$$
which satisfies \eqref{syssimpl} and \eqref{bcsysgen}. It is worth to notice that, setting
\begin{equation}\label{Dsigma}
D_{\sigma }:=\{ (x,y)\in \mathbb{R}^N \times \mathbb{R}^N : |x-y|< \sigma\} \quad (\sigma >0),
\end{equation}
if $u$ is a solution of \eqref{syssimpl}-\eqref{bcsysgen} then, since
$$|u(T+1)-u(0)| \leq \sum_{n=1}^{T+1}\left |  \Delta u(n-1) \right | <(T+1)\, a ,$$
actually one has that
\begin{equation}\label{velongs}
(u(0),u(T+1)) \in D(\gamma) \cap D_{(T+1)\, a}.
\end{equation}
The existence and uniqueness result which we obtain for problem \eqref{syssimpl}-\eqref{bcsysgen} (see Theorem \ref{thsolQ}) will enable us to discuss the solvability of a general system having the form
\begin{equation}\label{sysgen}
-\Delta\left[\phi(\Delta u(n-1))\right]=f(n, u(0), \dots , u(T+1)) \quad (n \in \mathbb{Z}[1,T]),
\end{equation}
associated with the boundary condition \eqref{bcsysgen}, under  hypotheses ($H_{\Phi}$), ($H_{\gamma}$) and, instead of ($H_{h}$),

\medskip
\noindent \hspace {0.3cm}$ (H_{f}) \; \; \;  f:\mathbb{Z}[1,T] \times (\mathbb{R}^N)^{T+2} \to \mathbb{R}^N  \mbox{\sl is continuous.}$

\medskip
\noindent Since the solutions in the cases of the systems \eqref{syssimpl}-\eqref{bcsysgen} and \eqref{sysgen}-\eqref{bcsysgen} are obtained by monotonicity and topological arguments, the question of variational solutions naturally arises. For this, however, it is necessary that the problem to have a potential structure. We will address a situation when such a structure occurs in the case of the system
\begin{equation}\label{potsysgendis}
-\Delta\left[\phi(\Delta u(n-1))\right]=\nabla_u F(n, u(n)) \quad (n \in \mathbb{Z}[1,T])
\end{equation}
associated with the boundary condition
\begin{equation}\label{potbcsysgen}
\left(\phi(\Delta u(0)),-\phi(\Delta u(T))\right)\in \partial j (u(0),u(T+1)),
\end{equation}
under hypothesis ($H_{\Phi}$), together with
$$\begin{array}{clrc}
(H_{F})  & F \! = \! F(n,u) \!  : \!  \mathbb{Z}[1,T] \times \mathbb{R}^N \to \mathbb{R} \, \, \mbox{\sl is continuous,  } \nabla_u F \, \, \mbox{\sl exists and  is continuous on the set}
 \; \mathbb{Z}[1,T] \times \mathbb{R}^N\\
 & \mbox{\sl and }   F(\, \cdot \, , 0_{\mathbb{R}^N})=0;\\
(H_{j})  & j:\mathbb{R}^N \times \mathbb{R}^N \rightarrow (-\infty, +\infty] \, \, \mbox{\sl is proper,  convex, lower semicontinuous,  }j(0_{\mathbb{R}^N \times \mathbb{R}^N})=0
 \; \mbox {\sl and } \\
 & 0_{\mathbb{R}^N \times \mathbb{R}^N}\in \partial j (0_{\mathbb{R}^N \times \mathbb{R}^N}),
\end{array}
$$
which respectively replace ($H_{f}$) and ($H_{\gamma}$). By $\partial j$  in \eqref{potbcsysgen} we have denoted the subdifferential of $j$ in the sense of convex analysis \cite{Rock}. It is useful to recall that if $K\subset \mathbb{R}^N \times \mathbb{R}^N$ is a nonempty, closed and convex set, then denoting by $I_K$ the indicator function of $K$ and by $N_K$ the normal cone to $K$, one has that $\partial I_K(z)=N_K(z)$ if $z\in K$ and  $\partial I_K(z)=\emptyset$ if $z\notin K$ (see p. 215-216 in \cite{Rock}).

\medskip
In recent years, much attention has been paid to problems with singular difference operators of
type $\Delta\left[\phi(\Delta u(n-1))\right]$,  perturbed with terms having various structures and associated with classical boundary
conditions such as Dirichlet, Neumann, Neumann-Steklov, periodic, subharmonic, etc. (see \cite{BeGh, BeMa, CaJeSe, DaGu, DMM, GrSe, LiLu, LuMa, LuMaLu, Ma1}  and the references therein). Such problems are discrete variants of the analogous differential ones and the prototype of $\Delta\left[\phi(\Delta u(n-1))\right]$ is the classical \textit{discrete relativistic operator}
$$u \mapsto \mathcal{R}(u):= \Delta\left[\frac{\Delta u(n-1)}{\sqrt{1-|\Delta u(n-1)|^2}}\right ] \qquad (n\in \mathbb{Z}[1,T]),$$
corresponding to $\phi (y)={y}/{\sqrt{1-|y|^2}}$ ($y\in B_1$). As emphasized in \cite{Ma1}, the operator $\mathcal{R}$  may be seen as a discretization of the acceleration in special relativity \cite{berg, LaLi}. Results on this topic focus on the cases when the unknown function is with either scalar values (most of them) \cite{BeMa, DaGu, DMM, LiLu, LuMa, LuMaLu} or,  as
in our case, is $\mathbb{R}^N$-valued \cite{BeGh, CaJeSe, GrSe, Ma1}. It is worth to notice that scalar unknown function ($N=1$)
corresponds to the discrete version of a scalar differential equation, while the more general case of $\mathbb{R}^N$-valued unknown function fits with discrete form of a differential system. Approaches to this topic are varied and make use of tools such as fixed point, topological degree, lower and upper solutions, critical point, as well as various combinations of these.

\medskip
It should be emphasized that hypothesis ($H_\phi$) -- also employed in \cite{GrSe, Ma1}, allows to consider general discrete singular $\phi$-Laplacians, among which we note the \textit{discrete $p$-relativistic operator}
$$u \mapsto \mathcal{R}_p(u):= \Delta\left[\frac{|\Delta u(n-1)|^{p-2}\Delta u(n-1)}{(1-|\Delta u(n-1)|^p)^{1-1/p}}\right ] \qquad (n\in \mathbb{Z}[1,T], \, p>1),$$
which is the discrete variant of the differential $p$-relativistic operator \cite{JeMaSe}, or its double phase variant $$u\mapsto \mathcal{R}_p(u)+\mathcal{R}_q(u) \quad (p,q>1,\,  p\neq q).$$
Observe that $\mathcal{R}_p$ is engendered by $\phi (y)={|y|^{p-2}y}/{(1-|y|^p)^{1-1/p}}$ ($y\in B_1$), has the potential $\Phi(y)=1-(1-|y|^p)^{1/p}$ ($y\in \overline{B}_1$) and yields $\mathcal{R}$ for $p=2$.

\medskip
Concerning the general boundary condition \eqref{bcsysgen}, first of all is worth to notice that it covers the classical ones. Thus, the Dirichlet, Neumann and mixed Dirichlet-Neumann homogeneous boundary
conditions
\begin{equation}\label{dirhdis}
 \quad u(0)= 0_{\mathbb{R}^N}=u(T+1) ,
\end{equation}
\begin{equation}\label{neuhdis}
\phi(\Delta u(0))= 0_{\mathbb{R}^N}=\phi( \Delta u(T)),
\end{equation}
\begin{equation}\label{mixedis}
 \quad u(0)= 0_{\mathbb{R}^N}=\phi( \Delta u(T)) ,
\end{equation}
are obtained by choosing $j=I_K$ with $K=\{ 0_{\mathbb{R}^N \times \mathbb{R}^N} \}$, $K=\mathbb{R}^N \times \mathbb{R}^N$ and $K=\left \{ (0_{\mathbb{R}^N} ,  \, x ) \; : \; x\in \mathbb{R}^N \right \}$, respectively. Denoting
$$d_N^1:=\left \{ (x,x) \, : \, x\in \mathbb{R}^N\right \}  \quad \mbox{ and } \quad d^2_N:=\left \{ (x,-x) \, : \, x\in \mathbb{R}^N\right \}, $$
the periodic and antiperiodic boundary conditions
\begin{equation}\label{peridis}
\quad u(0)-u(T+1)=0_{\mathbb{R}^N}=\phi(\Delta u(0))-\phi(\Delta u(T))  ,
\end{equation}
\begin{equation}\label{aperidis}
u(0)+u(T+1)= 0_{\mathbb{R}^N}=\phi(\Delta u(0))+\phi(\Delta u(T))
\end{equation}
are obtained with $K=d^1_N$ and  $K=d^2_N$, respectively. And this is not the only way to obtain conditions \eqref{peridis} and \eqref{aperidis}, these being also particular cases of a wider class. To provide this class, let $\mathbb{U}$ be an $N \times N$ real orthogonal matrix and consider the subspace
$Z= \left \{ (  x, \mathbb{U} \, x) \; : \; x\in \mathbb{R}^N \right \} \subset \mathbb{R}^N \times \mathbb{R}^N.$
Then we have
$$\partial I_Z(z)=N_Z(z)=Z^{\perp}=\{ (-\mathbb{U}^{\tau}x, x) \; : \; x \in \mathbb{R}^N \} \quad (z\in Z)$$
(here, $\mathbb{U}^{\tau}$ denotes the transpose of the matrix $\mathbb{U}$) and \eqref{potbcsysgen} becomes
\begin{equation}\label{Qperidis}
u(T+1) - \mathbb{U}\, u(0) =0_{\mathbb{R}^N}= \phi(\Delta u(T)) - \mathbb{U}\, \phi(\Delta u(0))  .
\end{equation}
A function $u : \mathbb{Z}[0,T+1] \to \mathbb{R}^N$ satisfying \eqref{Qperidis} is called rotating periodic and this kind of functions might be periodic ($\mathbb{U} = \mathbb{I}$ -- the identity matrix on $\mathbb{R}^N$),  subharmonic ($\mathbb{U}^k = \mathbb{I}$ for some $k \geq 2$) -- including antiperiodic ($\mathbb{U} = -\mathbb{I}$), or quasi-periodic ($\mathbb{U}^k \neq \mathbb{I}$ for any $k \geq 1$).
\medskip

More general, if $G:\mathbb{R}^N \times \mathbb{R}^N \to \mathbb{R}$ is a convex $C^1$
function with $G(0_{\mathbb{R}^N \times \mathbb{R}^N})=0$ and $G'(0_{\mathbb{R}^N \times \mathbb{R}^N})=0_{\mathbb{R}^N \times \mathbb{R}^N}$, and $K\subset \mathbb{R}^N \times \mathbb{R}^N$ is a closed and convex set with $0_{\mathbb{R}^N \times \mathbb{R}^N} \in K$, then taking $j=G+I_K$, we have that $j$ satisfies $(H_{j})$ and \eqref{potbcsysgen} reads
\begin{equation}\label{ncon}
\left ( u(0), u(T+1) \right ) \in K, \qquad \left ( \phi ( \Delta u(0)), -\phi(\Delta u(T)) \right ) -G' (u(0),u(T+1))\in N_K(u(0), u(T+1)).
\end{equation}

\medskip
It is worth emphasizing the novelty of the fact that this study is one that treats in a unified manner a class of problems that is wide from the point of view of the generality of the discrete singular $\phi$-Laplacian, of the nonlinear perturbations in the system, as well as of the form of the boundary condition \eqref{bcsysgen} and its particular case \eqref{potbcsysgen}. We note that a general condition of the type \eqref{potbcsysgen} was used for problems with the discrete $p$-Laplacian \cite{cSe} and with the discrete $p(\cdot )$-Laplacian \cite{BeJeSe1}. An analysis of the approach and techniques in this study, compared to those in \cite{BeJeSe1, cSe} highlights major differences brought about by the singular character of the discrete $\phi$-Laplacian operator.

\medskip

The rest of the paper is organized as follows. In Section 2 we prove an existence and uniqueness result for problem \eqref{syssimpl}-\eqref{bcsysgen} (Theorem \ref{thsolQ}). For this we use monotonicity and topological fixed point arguments. It is worth mentioning that in the approaches from the next two sections, Theorem \ref{thsolQ} is a key ingredient. Section 3 is devoted to problem \eqref{sysgen}-\eqref{bcsysgen}. Here we obtain the existence of solutions by a priori estimates. In Section 4 we provide a variational formulation of problem \eqref{potsysgendis}-\eqref{potbcsysgen} in the frame of critical point theory for convex, lower semicontinuous perturbations of $C^1$-functionals, developed by Szulkin \cite{Sz}. Next, the last section deals with existence of minimum energy and saddle-point solutions for a problem of the type \eqref{potsysgendis}-\eqref{potbcsysgen}.

\section{A unique solvability result}
Let $\gamma:\mathbb{R}^N\times\mathbb{R}^N\to2^{\mathbb{R}^N\times\mathbb{R}^N}$ be a set-valued monotone operator. In this section, we assume ($H_{\Phi}$), ($H_{h}$) and we are concerned with solvability of problem \eqref{syssimpl}-\eqref{bcsysgen}, which we conveniently write:
\begin{equation*}\label{bvp}
  [Q_\gamma(h)]\qquad\qquad
  \left\{
  \begin{array}{lcl}
    -\Delta\left[\phi(\Delta u(n-1))\right]+u(n)=h(n) & \hbox{$\quad (n \in \mathbb{Z}[1,T]),$} \\
    \left(\phi(\Delta u(0)),-\phi(\Delta u(T))\right)\in\gamma(u(0),u(T+1)),
  \end{array}
  \right.
  \end{equation*}
As we will see, invoking the hypothesis ($H_{\gamma}$) will only be necessary in the main result at the end of the section (Theorem \ref{thsolQ}).
We will make use of the following spaces
$$X_T:=\{ u:\mathbb{Z}[1,T]\to \mathbb{R}^N \} \quad \mbox{and} \quad  X_{T+2}:=\{ u:\mathbb{Z}[0,T+1]\to \mathbb{R}^N \},$$
which are respectively endowed with the inner products
$$(u|v)_T=\sum_{j=1}^{T}\langle u(j)|v(j) \rangle, \qquad (u|v)_{T+2}=\sum_{j=0}^{T+1}\langle u(j)|v(j) \rangle$$
and the corresponding norms
$$\|u\|_T=\left( \sum_{j=1}^{T}|u(j)|^2  \right)^{{1}/{2}}, \qquad   \|u\|_{T+2}=\left( \sum_{j=0}^{T+1}|u(j)|^2  \right)^{{1}/{2}}.$$
For $u,v \in X_{T+2}$  with $\|\Delta u \|_{\infty} <a$ and $\displaystyle \|\Delta v \|_{\infty} <a$, set
        \begin{align*}
        \omega(u&,v):=\langle\phi(\Delta u(0))-\phi(\Delta v(0))|u(0)-v(0)\rangle-\langle\phi(\Delta u(T))-\phi(\Delta v(T))|u(T+1)-v(T+1)\rangle \\
                &=\langle\!\langle\left(\phi(\Delta u(0)),-\phi(\Delta u(T))\right)-\left(\phi(\Delta v(0)),-\phi(\Delta v(T))\right)|\left(u(0),u(T+1)\right)-\left(v(0),v(T+1)\right)\rangle\!\rangle,
        \end{align*}
$${\mathcal O}(u,v):=-\sum_{n=1}^{T}\left\langle\Delta\left[\phi(\Delta u(n-1))\right]-\Delta\left[\phi(\Delta v(n-1))\right]|\ u(n)-v(n)\right\rangle,$$
$${\mathcal M}(u,v):=\sum_{n=1}^{T+1}\left\langle\phi(\Delta u(n-1))-\phi(\Delta v(n-1))|\Delta u(n-1)-\Delta v(n-1)\right\rangle$$
and observe that, as $\phi$ is (strictly) monotone, one has ${\mathcal M}(u,v) \geq 0$. Also, if $u,v$ satisfy the boundary condition in problem $[Q_\gamma(h)]$, that is
$$\left(\phi(\Delta u(0)),-\phi(\Delta u(T))\right)\in\gamma(u(0),u(T+1)) \quad \mbox{and} \quad  \left(\phi(\Delta v(0)),-\phi(\Delta v(T))\right)\in\gamma(v(0),v(T+1)),$$
then, by the monotonicity of $\gamma$, one has $\omega(u,v) \geq 0$.
\smallskip

\begin{proposition}\label{prop21}
If $u,v \in X_{T+2}$ are with $\|\Delta u \|_{\infty}<a$ and $\|\Delta v \|_{\infty}<a$, then  it holds
\begin{equation}\label{fsp}
{\mathcal O}(u,v)=\omega(u,v)+{\mathcal M}(u,v).
\end{equation}
\end{proposition}

\noindent\proof Using the summation by parts formula
\begin{equation}\label{sbpf}
         \sum_{n=1}^{T}\langle\Delta \xi(n-1)|\eta(n)\rangle=\langle \xi(T)|\eta(T)\rangle-\langle \xi(0)|\eta(0)\rangle-\sum_{n=1}^{T}\langle \xi(n-1)|\Delta \eta(n-1)\rangle \quad (\xi, \eta \in X_{T+2}),
         \end{equation}
we obtain
\begin{align*}
    {\mathcal O}(u,v)=\ & -\sum_{n=1}^{T}\left\langle\Delta\left[\phi(\Delta u(n-1))-\phi(\Delta v(n-1))\right]|\ u(n)-v(n)\right\rangle\\
    =\ &-\langle\phi(\Delta u(T))-\phi(\Delta v(T))|\ u(T)-v(T)\rangle\\
    &+\langle\phi(\Delta u(0))-\phi(\Delta v(0))|\ u(0)-v(0)\rangle\\
    &+\sum_{n=1}^{T}\left\langle\phi(\Delta u(n-1))-\phi(\Delta v(n-1))|\Delta(u-v)(n-1)\right\rangle\\
    =\ &\sum_{n=1}^{T+1}\left\langle\phi(\Delta u(n-1))-\phi(\Delta v(n-1))|\Delta(u-v)(n-1)\right\rangle\\
    &-\langle\phi(\Delta u(T))-\phi(\Delta v(T))|\ \Delta(u-v)(T)\rangle\\
    &-\langle\phi(\Delta u(T))-\phi(\Delta v(T))|\ u(T)-v(T)\rangle\\
    &+\langle\phi(\Delta u(0))-\phi(\Delta v(0))|\ u(0)-v(0)\rangle\\
    =\ &{\mathcal M}(u,v) +\omega(u,v).
\end{align*} \cqfd

\begin{remark}\label{remMo}
\noindent \emph{Let $u,v$ be solutions of $[Q_\gamma(h)]$, respectively $[Q_\gamma(l)]$, ($h,l\in X_T$). Multiplying
$$-\Delta\left[\phi(\Delta u(n-1))\right]+\Delta\left[\phi(\Delta v(n-1))\right]+u(n)-v(n)=h(n)-l(n)\qquad (n \in \mathbb{Z}[1,T])$$
by $(u(n)-v(n))$ and summing between $1$ and $T$, from \eqref{fsp} we infer
\begin{equation}\label{est}
0\leq {\mathcal M}(u,v)+\omega(u,v)+\|u-v\|_T^2=\sum_{n=1}^T\langle h(n)-l(n) \, |\, u(n)-v(n) \rangle \leq \sqrt{T} |h-l |_{\infty}\|u-v\|_T
\end{equation}
\noindent(here and hereafter, for $\xi\in X_T$, we have denoted by $|\xi|_{\infty}$ the sup-norm on $X_T$, that is $|\xi|_{\infty}:=\displaystyle\max_{n\in \mathbb{Z}[1,T]} \, |\xi(n)|$).}

\emph{If $0_{\mathbb{R}^N \times \mathbb{R}^N}\in \gamma(0_{\mathbb{R}^N \times \mathbb{R}^N})$, then $0_{X_{T+2}}$ is a solution of $[Q_\gamma(0_{X_T})]$ and \eqref{est} becomes
\begin{equation}\label{est1}
0\leq {\mathcal M}(u,0_{X_{T+2}})+\omega(u,0_{X_{T+2}})+\|u\|_T^2=\sum_{n=1}^T\langle h(n) \, |\, u(n) \rangle \leq \sqrt{T} |h |_{\infty}\|u\|_T.
\end{equation}}
\end{remark}

\begin{proposition}\label{propunic}
Problem $[Q_\gamma(h)]$ has at most one solution.
\end{proposition}

\noindent\proof If $u$ and $v$ are two solutions of $[Q_\gamma(h)]$, then from \eqref{est} with $l=h$, it follows
$${\mathcal M}(u,v)+\|u-v\|_T^2=0,$$
which implies $u=v$, by the strict monotonicity of $\phi$.\cqfd
\medskip

\begin{remark}
\emph{Observe that  Proposition \ref{propunic}  remains valid if instead of the entire hypothesis $(H_{\Phi})$ we solely assume that "\textit{$\phi:B_a \to \mathbb{R}^N$ is strictly monotone}".}
\end{remark}
\medskip

Next, let us consider the non-homogeneous Dirichlet and Neumann problems

\begin{equation*}\label{dirxy}
  [\mathcal{D}_{x,y}(h)]\qquad\qquad
  \left\{
  \begin{array}{lcl}
    -\Delta\left[\phi(\Delta u(n-1))\right]+u(n)=h(n) & \hbox{$\quad (n \in \mathbb{Z}[1,T]),$} \\
    u(0)=x,\ \ u(T+1)=y,
  \end{array}
  \right.
  \end{equation*}
respectively,
\begin{equation*}\label{neuxy}
  [\mathcal{N}_{x,y}(h)]\qquad\qquad
  \left\{
  \begin{array}{lcl}
    -\Delta\left[\phi(\Delta u(n-1))\right]+u(n)=h(n) & \hbox{$\quad (n \in \mathbb{Z}[1,T]),$} \\
    \phi(\Delta u(0))=x,\ \ \phi(\Delta u(T))=y,
  \end{array}
  \right.
\end{equation*}
where $x,y\in \mathbb{R}^N$ and $h\in X_T$.

\begin{proposition}\label{teu} (i) If $(x,y)\in D_{(T+1)a}$ then problem  $[\mathcal{D}_{x,y}(h)]$ has a unique solution $u_{x,y}$.

\medskip
(ii) For any $x,y\in \mathbb{R}^N$, problem $[\mathcal{N}_{x,y}(h)]$ has a unique solution $\widehat{u}_{x,y}$.
\end{proposition}

\noindent\proof For the uniqueness part, observe that $[\mathcal{D}_{x,y}(h)]$ and $[\mathcal{N}_{x,y}(h)]$ are problems of type $[Q_\gamma(h)]$. Indeed, for problem $[\mathcal{D}_{x,y}(h)]$ one take $D(\gamma)=\{ (x,y)\}$, $\gamma (x,y)= \mathbb{R}^N \times \mathbb{R}^N$, while
$[\mathcal{N}_{x,y}(h)]$ corresponds to $D(\gamma)=\mathbb{R}^N \times \mathbb{R}^N$, $\gamma (\xi,\eta)= \{(x,-y)\}$ for all $(\xi, \eta) \in \mathbb{R}^N \times \mathbb{R}^N$. Thus, in both cases, Proposition \ref{propunic} applies. On account of \cite[Theorem 2.4]{GrSe}, problem $[\mathcal{D}_{x,y}(h)]$ has a solution if and only if $(x,y)\in D_{(T+1)a}$. The fact that $[\mathcal{N}_{x,y}(h)]$ is solvable follows by \cite[Corollary 3.5]{BeGh}.\cqfd

Below, we keep the notations in Proposition \ref{teu}, namely $\widehat{u}_{x,y}$ denotes the solution of $[\mathcal{N}_{x,y}(h)]$ and, if condition $(x,y)\in D_{(T+1)a}$ is satisfied, $u_{x,y}$ stands for the solution of $[\mathcal{D}_{x,y}(h)]$.

\begin{proposition}\label{ubarcont}
The mapping $\mathbb{R}^N \times \mathbb{R}^N \ni (x,y)  \mapsto (\widehat{u}_{x,y}(0), \widehat{u}_{x,y}(T+1)) \in \mathbb{R}^N \times \mathbb{R}^N$ is continuous.
\end{proposition}

\noindent\proof Let $(x,y)\in \mathbb{R}^N \times \mathbb{R}^N$. Summing
\begin{equation}\label{ubc1}
-\Delta\left[\phi(\Delta \widehat{u}_{x,y}(n-1))\right]+\widehat{u}_{x,y}(n)=h(n)\quad (n \in \mathbb{Z}[1,T])
\end{equation}
between 1 and $T$, and using the boundary conditions $\phi(\Delta \widehat{u}_{x,y}(0))=x$ and $\phi(\Delta \widehat{u}_{x,y}(T))=y$, we obtain
\begin{equation}\label{ubc2}
\sum_{k=1}^T\widehat{u}_{x,y}(k)=y-x+\sum_{k=1}^T h(k).
\end{equation}
Also, summing again in \eqref{ubc1} between 1 and $k\in\mathbb{Z}[1,T]$, one has
$$\phi(\Delta \widehat{u}_{x,y}(k))-x=\sum_{j=1}^k\left(\widehat{u}_{x,y}(j)-h(j)\right)\quad (k \in \mathbb{Z}[1,T]),$$
which yield
$$\Delta \widehat{u}_{x,y}(k)=\phi^{-1}\left(x+\sum_{j=1}^k\left(\widehat{u}_{x,y}(j)-h(j)\right)\right)\quad (k \in \mathbb{Z}[0,T]).$$
From this it is straightforward to see that
\begin{equation}\label{ubc3}
\widehat{u}_{x,y}(i)=\widehat{u}_{x,y}(0)+\displaystyle\sum_{l=0}^{i-1} \phi^{-1}\left(x+\sum_{j=1}^l\left(\widehat{u}_{x,y}(j)-h(j)\right)\right)\quad (i \in \mathbb{Z}[0,T+1]).
\end{equation}
Then, combining \eqref{ubc3} with \eqref{ubc2}, it follows
\begin{equation}\label{ubar0}
\widehat{u}_{x,y}(0)=\frac{1}{T} \left [y-x+ \sum_{i=1}^T h(i) - \sum_{i=1}^T \sum_{l=0}^{i-1} \phi^{-1}\left(x+\sum_{j=1}^l\left(\widehat{u}_{x,y}(j)-h(j)\right)\right)\right].
\end{equation}
Also, from \eqref{ubc3}, we get
\begin{equation}\label{ubarTp1}
\widehat{u}_{x,y}(T+1)=\widehat{u}_{x,y}(0)+\displaystyle\sum_{l=0}^{T} \phi^{-1}\left(x+\sum_{j=1}^l\left(\widehat{u}_{x,y}(j)-h(j)\right)\right).
\end{equation}

Now, let $(x_k,y_k) \rightarrow (x,y)$ in $\mathbb{R}^N \times \mathbb{R}^N$, as $k \to \infty$. Using \eqref{ubar0} and \eqref{ubarTp1} and the boundedness of the range of $\phi^{-1}$ we deduce that sequences $\{\widehat{u}_{x_k,y_k}(0)\}$
and $\{\widehat{u}_{x_k,y_k}(T+1)\}$ are bounded in $\mathbb{R}^N$. Multiplying
$$-\Delta\left[\phi(\Delta \widehat{u}_{x_k,y_k}(n-1))\right]+\Delta\left[\phi(\Delta \widehat{u}_{x,y}(n-1))\right]+\widehat{u}_{x_k,y_k}(n)-\widehat{u}_{x,y}(n)=0\qquad (n \in \mathbb{Z}[1,T])$$
by $(\widehat{u}_{x_k,y_k}(n)-\widehat{u}_{x,y}(n))$ and summing between $1$ and $T$, from \eqref{fsp} we infer
\begin{equation*}
0={\mathcal O}(\widehat{u}_{x_k,y_k},\widehat{u}_{x,y})+\|\widehat{u}_{x_k,y_k}-\widehat{u}_{x,y}\|_T^2=
\omega(\widehat{u}_{x_k,y_k},\widehat{u}_{x,y})+{\mathcal M}(\widehat{u}_{x_k,y_k},\widehat{u}_{x,y})+\|\widehat{u}_{x_k,y_k}-\widehat{u}_{x,y}\|_T^2,
\end{equation*}
which, since ${\mathcal M}(\widehat{u}_{x_k,y_k},\widehat{u}_{x,y})\geq0$, gives
$$\|\widehat{u}_{x_k,y_k}-\widehat{u}_{x,y}\|_T^2\leq\langle y_k-y\ |\ \widehat{u}_{x_k,y_k}(T+1)-\widehat{u}_{x,y}(T+1)\rangle-
\langle x_k-x\ |\ \widehat{u}_{x_k,y_k}(0)-\widehat{u}_{x,y}(0)\rangle$$
implying, by the boundedness of $\{\widehat{u}_{x_k,y_k}(0)\}$
and $\{\widehat{u}_{x_k,y_k}(T+1)\}$, that $\widehat{u}_{x_k,y_k} \rightarrow \widehat{u}_{x,y}$ in $X_T$, as $k \to \infty$.
\medskip

Finally, as $\phi(\Delta \widehat{u}_{x_k,y_k}(0))=x_k$ and $\widehat{u}_{x_k,y_k}(1) \rightarrow \widehat{u}_{x,y}(1)$, one has
$$\widehat{u}_{x_k,y_k}(0)=\widehat{u}_{x_k,y_k}(1)-\phi^{-1}(x_k)\rightarrow\widehat{u}_{x,y}(1)-\phi^{-1}(x)=\widehat{u}_{x,y}(0),$$
as $k \to \infty$. Similar, we get $\widehat{u}_{x_k,y_k}(T+1) \rightarrow \widehat{u}_{x,y}(T+1)$ as $k \to \infty$ and the proof is complete.\cqfd
\medskip

\begin{lemma}\label{lemarmj}
If $w\in X_{T+2}$ is with $\displaystyle \| \Delta w \|_\infty\leq a$ ($\mbox{resp.}< a$), then for all $m, i\in \mathbb{Z}[0,T+1]$, there exists $r_{m,i}\in\mathbb{R}^N$ with $|r_{m,i}|\leq  |m-i|a$
($\mbox{resp.}< |m-i|a$ if $m\neq i$) such that $w(m)=w(i)+r_{m,i}$.
\end{lemma}

\noindent\proof If $m>i$ we choose $r_{m,i}=\displaystyle \sum_{n=i+1}^m\Delta w(n-1)$,  while if $m<i$ we take $r_{m,i}=\displaystyle -\sum_{n=m+1}^i\Delta w(n-1)$. \cqfd

\begin{proposition}\label{estni}
Any $w\in X_{T+2}$ with $\displaystyle \| \Delta w \|_\infty\leq a$ satisfies
\begin{equation}\label{estimarew}
|w(m)|\leq\frac{1}{\sqrt{T}}\, \|w\|_T+Ta\qquad (m\in \mathbb{Z}[0,T+1]).
\end{equation}
\end{proposition}

\noindent\proof It is immediate from Lemma \ref{lemarmj}.\cqfd

\medskip

 Now, let $\theta : \mathbb{R}^N \times \mathbb{R}^N \to \mathbb{R}^N \times \mathbb{R}^N$ be defined by
$$D(\theta):=D_{(T+1)a},\qquad \theta (x,y):= \left(-\phi(\Delta u_{x,y}(0)),\ \phi(\Delta u_{x,y}(T))\right)\quad \left((x,y)\in D_{(T+1)a}\right).$$

\begin{proposition}\label{optheta}
The operator $\theta$ is maximal monotone and coercive.
\end{proposition}

\noindent\proof According to \cite[Proposition 2.2]{Br}, in order to show the maximal monotonicity of operator $\theta$ we have to prove that: (i) $\theta$ is monotone and (ii) $\theta + i_d:D_{(T+1)a} \to \mathbb{R}^N \times \mathbb{R}^N$ is surjective (here and hereafter, $i_d$ stands for the identity mapping on $\mathbb{R}^N \times \mathbb{R}^N$).
\medskip

\noindent (i) Let $(x,y), \, (\xi , \eta)\in D_{(T+1)a}$ and $u_{x,y}$, $u_{\xi,\eta}$ be solutions of $[\mathcal{D}_{x,y}(h)]$, respectively $[\mathcal{D}_{\xi,\eta}(h)]$. We have to show that
$$ \langle \! \langle \theta (x,y) - \theta (\xi, \eta)| (x,y) - (\xi, \eta) \rangle \! \rangle \geq 0,$$
which, if we make use of the boundary conditions ($u_{x,y}(0)=x, \, u_{x,y}(T+1)=y, \, u_{\xi , \eta}(0)=\xi, \, u_{\xi , \eta}(T+1)=\eta $), means
\begin{equation}\label{mtheta}
-\omega(u_{x,y},u_{\xi,\eta})=-\langle \phi(\Delta u_{x,y}(0))-\phi(\Delta u_{\xi,\eta}(0)) | x-\xi \rangle+\langle \phi(\Delta u_{x,y}(T))-\phi(\Delta u_{\xi,\eta}(T)) | y-\eta \rangle \geq 0.
\end{equation}
Multiplying
$$-\Delta\left[\phi(\Delta u_{x,y}(n-1))\right]+\Delta\left[\phi(\Delta u_{\xi,\eta}(n-1))\right]+u_{x,y}(n)-u_{\xi,\eta}(n)=0\qquad (n \in \mathbb{Z}[1,T])$$
by $(u_{x,y}(n)-u_{\xi,\eta}(n))$ and summing between $1$ and $T$, from \eqref{fsp}, one get
\begin{equation*}
0={\mathcal O}(u_{x,y},u_{\xi,\eta})+\|u_{x,y}-u_{\xi,\eta}\|_T^2=
\omega(u_{x,y},u_{\xi,\eta})+{\mathcal M}(u_{x,y},u_{\xi,\eta})+\|u_{x,y}-u_{\xi,\eta}\|_T^2,
\end{equation*}
and, as ${\mathcal M}(u_{x,y},u_{\xi,\eta})\geq0$, clearly \eqref{mtheta} holds true.
\medskip

\noindent (ii) Let $(\xi, \eta)\in {\mathbb{R}^N \times \mathbb{R}^N}$ be arbitrary given. We have to prove that there exists some $(x,y)= (x_{\xi, \eta}, y_{\xi, \eta})\in D_{(T+1)a}$ such that
$$\theta(x,y) + i_d(x,y)=\left(-\phi(\Delta u_{x,y}(0)),\ \phi(\Delta u_{x,y}(T))\right)+(x,y)=(\xi,\eta),$$
i.e.,
\begin{equation}\label{stheta}
\left(-\phi(\Delta u_{x,y}(0)),\ \phi(\Delta u_{x,y}(T))\right)=\left(\xi-x,\eta - y \right).
\end{equation}
With this aim, we proceed by a fixed point argument, as follows. Let $\Lambda = \Lambda _{\xi , \eta} : \mathbb{R}^N \times \mathbb{R}^N \to \mathbb{R}^N \times \mathbb{R}^N$ be the operator defined by
$$\Lambda(x,y)= \left ( \widehat{u}_{x-\xi, \eta-y} (0), \widehat{u}_{x-\xi, \eta-y} (T+1) \right ) \quad \left ( (x,y) \in \mathbb{R}^N \times \mathbb{R}^N \right );$$
recall $\widehat{u}_{x-\xi, \eta-y}$ is the unique solution of problem $[\mathcal{N}_{x-\xi,\eta-y}(h)]$.
Suppose that we have already proved that there exists a fixed point $(x,y)$ of $\Lambda$, i.e.,
\begin{equation*}
\left ( \widehat{u}_{x-\xi, \eta-y} (0), \widehat{u}_{x-\xi, \eta-y} (T+1) \right )=(x,y),
\end{equation*}
which means that the solution $\widehat{u}_{x-\xi, \eta-y}$ satisfies the Dirichlet boundary conditions
\begin{equation*}
\widehat{u}_{x-\xi, \eta-y}(0)=x, \quad \widehat{u}_{x-\xi, \eta-y}(T+1)=y.
\end{equation*}
Thus, by Proposition \ref{teu} $(i)$, we have that $(x,y)\in D_{(T+1)a}$ and $\widehat{u}_{x-\xi, \eta-y}=u_{x,y}$, which implies \eqref{stheta}.
\medskip

Hence, it remains to prove that $\Lambda$ has a fixed point.
From Proposition \ref{ubarcont} it is immediate that operator $\Lambda$ is continuous and hence compact on the finite dimensional space $\mathbb{R}^N \times \mathbb{R}^N$. By virtue of Schaefer's theorem (see e.g. \cite[Corollary 4.4.12]{Llo}), it suffices to show that there is a positive constant $c$ such that
\begin{equation}\label{constc}
|x|+|y| \leq c,
\end{equation}
for all $(x,y) \in \mathbb{R}^N \times \mathbb{R}^N$ satisfying $(x,y)=\mu \, \Lambda (x,y)$ with some $\mu\in (0,1]$.
\medskip

So, let $(x,y) \in \mathbb{R}^N \times \mathbb{R}^N$ be such that
\begin{equation*}\label{aest}
(x,y)= \mu \left(\widehat{u}_{x-\xi, \eta-y}(0), \widehat{u}_{x-\xi, \eta-y}(T+1)\right)
\end{equation*}
with $\mu\in (0,1]$. Using \eqref{ubar0} and \eqref{ubarTp1}, we obtain
\begin{equation}\label{sysxy}
\left ( \frac{1}{\mu}+\frac{1}{T} \right ) x +\frac{1}{T} \, y= Q_1(x,y)\quad \mbox{and}\quad y=x+\mu Q_2(x,y),
\end{equation}
where
$$Q_1(x,y):=\frac{1}{T} \left [ \xi + \eta + \sum_{i=1}^T h(i) -\sum_{i=1}^T \sum_{l=0}^{i-1} \phi^{-1}\left(x-\xi+\sum_{j=1}^l\left(\widehat{u}_{x-\xi,\eta-y}(j)-h(j)\right)\right)\right ],$$
$$Q_2(x,y):=\sum_{l=0}^{T} \phi^{-1}\left(x-\xi+\sum_{j=1}^l\left(\widehat{u}_{x-\xi,\eta-y}(j)-h(j)\right)\right).$$
We have the estimates
\begin{equation}\label{Q1Q2}
|Q_1(x,y)|<\frac{1}{T}\ \overline{Q}+Ta\quad \mbox{and}\quad |Q_2(x,y)|<(T+1)a,
\end{equation}
with $\overline{Q}:=|\xi|+|\eta|+T|h|_\infty$. Next, solving \eqref{sysxy}, we get
\begin{equation}\label{sysxysol}
x=\frac{\mu T Q_1(x,y)-\mu^2 Q_2(x,y)}{T+2\mu}\quad \mbox{and}\quad y=\frac{\mu T Q_1(x,y)-\mu^2 Q_2(x,y)}{T+2\mu}+\mu Q_2(x,y).
\end{equation}
Then, using \eqref{Q1Q2}, \eqref{sysxysol}, together with $\mu\in (0,1]$, we infer that
$$|x|\leq\frac{\overline{Q}+(T^2+T+1)a}{T}\quad \mbox{and}\quad |y|\leq\frac{\overline{Q}+(T^2+T+1)a}{T}+(T+1)a.$$
Therefore, \eqref{constc} holds true with $$c=2\frac{\overline{Q}+(T^2+T+1)a}{T}+(T+1)a$$
and the proof of the maximal monotonicity of $\theta$ is complete.
\medskip

We show now that $\theta$ is coercive, i.e.,
\begin{equation}\label{coerciv}
\lim_{\scriptsize \begin{array}{crlc}
&\|(x,y)\| \to +\infty \\
&(x,y) \in D_{(T+1)a}
\end{array}} \! \!
{\frac{\langle \! \langle \theta(x,y)|(x,y)\rangle \! \rangle }{\|(x,y)\|}=+\infty}.
\end{equation}
Let $(x,y)\in D_{(T+1)a}$. As $u_{x,y}$ is solution of the Dirichlet problem $[\mathcal{D}_{x,y}(h)]$, from
$$\mathcal{O}(u_{x,y},0_{X_{T+2}})+\|u_{x,y}\|_T^2=\sum_{n=1}^T\langle h(n)|\, u_{x,y}(n)\rangle$$
and Proposition \ref{prop21} it follows
$$\omega(u_{x,y},0_{X_{T+2}} )+\mathcal{M}(u_{x,y},0_{X_{T+2}})+\|u_{x,y}\|_T^2=\sum_{n=1}^T\langle h(n)|\, u_{x,y}(n)\rangle\leq|h|_\infty\sqrt{T}\|u_{x,y}\|_T.$$
On the other hand, using the boundary conditions ($u_{x,y}(0)=x$, $u_{x,y}(T+1)=y$), we infer that
$$\langle \! \langle \theta(x,y)|(x,y)\rangle \! \rangle=-\omega(u_{x,y},0_{X_{T+2}})$$
and, as $\mathcal{M}(u_{x,y},0_{X_{T+2}})\geq0$, one get
\begin{equation}\label{thetaest1}
\langle \! \langle \theta(x,y)|(x,y)\rangle \! \rangle\geq\|u_{x,y}\|_T\left(\|u_{x,y}\|_T-|h|_\infty\sqrt{T}\right).
\end{equation}
From \eqref{estimarew}, it follows
\begin{equation*}
|u_{x,y}(n)| \leq \frac{1}{\sqrt{T}}\, \|u_{x,y}\|_T+T a \qquad (n\in\mathbb{Z}[0,T+1]).
\end{equation*}
Since $x=u_{x,y}(0)$, $y=u_{x,y}(T+1)$, this yields
$$\max\left\{|x|,\, |y|\right\} \leq \frac{1}{\sqrt{T}}\, \|u_{x,y}\|_T+T a$$
and, hence
$$\|(x,y)\|= \sqrt{|x|^2+|y|^2}\leq \sqrt{2}\left(\frac{1}{\sqrt{T}}\, \|u_{x,y}\|_T+T a \right),$$
or
\begin{equation*}
\|u_{x,y}\|_T \geq \sqrt{T} \left(\frac{1}{\sqrt{2}}\, \|(x,y)\|-Ta \right).
\end{equation*}
Then, for $(x,y)\in D_{(T+1)a}$ with $\|(x,y)\|>\sqrt{2} (Ta +|h|_{\infty})$, this together with \eqref{thetaest1} give
$$ \displaystyle \langle \! \langle \theta(x,y)|(x,y)\rangle \! \rangle \geq T\left(\frac{1}{\sqrt{2}}\|(x,y)\|-Ta\right)\left (\frac{1}{\sqrt{2}}\|(x,y)\|-Ta-|h|_{\infty}\right),$$
which clearly implies \eqref{coerciv} and the proof is complete.\cqfd
\medskip

Now, we can state the existence and uniqueness result for problem $[Q_\gamma (h)]$.

\begin{theorem}\label{thsolQ}
If $\gamma$ satisfies ($H_{\gamma}$), then problem $[Q_\gamma (h)]$ has a unique solution $u_h\in X_{T+2}$, for any $h\in X_T$.
\end{theorem}

\noindent\proof
The proof follows the outline of the one in \cite[Theorem 2.8]{Je}; we give a sketch for the convenience of the reader. Thus,
Proposition \ref{propunic} ensures the uniqueness part. For the existence part, let $\Gamma : \mathbb{R}^N \times \mathbb{R}^N \to 2^{\mathbb{R}^N \times \mathbb{R}^N}$ be defined by
$$D(\Gamma):=D(\gamma) \cap D_{(T+1)a}, \qquad \Gamma (x,y):= \gamma (x,y)+ \theta (x,y)\quad \left((x,y)\in D(\gamma)\cap D_{(T+1)a}\right).$$
From \cite[Corollaire 2.7]{Br}, taking into account that $\theta$ is maximal monotone by Proposition \ref{optheta} and $$D(\gamma) \cap \mbox{int}\, D(\theta)=D(\gamma)\cap D_{(T+1)a}\ni 0_{\mathbb{R}^N \times \mathbb{R}^N},$$ we have that
$\Gamma$ is maximal monotone.
\medskip

Next, we show that $\Gamma: D(\Gamma) \to 2^{\mathbb{R}^N \times \mathbb{R}^N}$ is surjective. This follows by \cite[Corollaire 2.2]{Br} in the case when $D(\Gamma)$ is bounded. If this is not the case, Corollaire 2.4 in \cite{Br} (also, see \cite[Corollary 32.35]{Ze}) will ensure the surjectivity, provided that operator $\Gamma$ is coercive, that is
\begin{equation}\label{coercivGam}
\lim_{\scriptsize \begin{array}{crlc}
&\|(x,y)\| \to +\infty \\
&(x,y) \in D(\Gamma)
\end{array}} \! \!
{\frac{\displaystyle \inf \left \{ \langle \! \langle (\xi,\eta)|(x,y)\rangle \! \rangle\,:\, (\xi, \eta ) \in \Gamma (x,y) \right \} }{\|(x,y)\|}=+\infty}.
\end{equation}

So, let $(x,y)\in D(\Gamma)$ and $(\xi,\eta)\in\Gamma(x,y)$. Taking $(\zeta,\nu)\in\gamma(x,y)$ such that
$(\xi,\eta )=(\zeta,\nu)+\theta(x,y)$ and using the fact that $\gamma$ is monotone and $0_{\mathbb{R}^N\times\mathbb{R}^N}\in\gamma(0_{\mathbb{R}^N \times\mathbb{R}^N})$, we obtain
$$\langle \! \langle (\xi,\eta)|(x,y)\rangle \! \rangle= \langle \! \langle (\zeta,\nu)-0_{\mathbb{R}^N \times \mathbb{R}^N}|(x,y)-0_{\mathbb{R}^N \times \mathbb{R}^N}\rangle \! \rangle + \langle \! \langle \theta(x,y)|(x,y)\rangle \! \rangle \geq \langle \! \langle \theta(x,y)|(x,y)\rangle \! \rangle,$$
which implies
$$\displaystyle \inf \left \{ \langle \! \langle (\xi,\eta)|(x,y)\rangle \! \rangle\, : \, (\xi, \eta ) \in \Gamma (x,y) \right \} \geq  \langle \! \langle \theta(x,y)|(x,y)\rangle \! \rangle.$$
This together with the coercivity of $\theta$ (Proposition \ref{optheta}), yields \eqref{coercivGam}.
\medskip

Now, from the surjectivity of $\Gamma$, there exists $(x,y)\in D(\gamma) \cap D_{(T+1)a}$ such that
$$0_{\mathbb{R}^N \times \mathbb{R}^N}\in\Gamma(x,y)=\gamma(x,y)+\theta(x,y)=\gamma(x,y)+\left(-\phi(\Delta u_{x,y}(0)),\ \phi(\Delta u_{x,y}(T))\right),$$
i.e.,
$$\left(\phi(\Delta u_{x,y}(0)),\ -\phi(\Delta u_{x,y}(T))\right)\in \gamma(x,y)=\gamma(u_{x,y}(0), u_{x,y}(T+1))$$
and the proof is accomplished by taking $u_h=u_{x,y}$. \cqfd
\bigskip

\section{Existence of solutions by a priori estimates}

In this section we assume that hypotheses ($H_{\Phi}$), ($H_{\gamma}$) are fulfilled and we are mainly interested in the solvability of the general problem \eqref{sysgen}-\eqref{bcsysgen}, that is
\begin{equation}\label{bvpf}
  \left\{
  \begin{array}{lcl}
    -\Delta\left[\phi(\Delta u(n-1))\right]=f(n,u(0),u(1),\ldots,u(T+1)) & \hbox{$\quad (n \in \mathbb{Z}[1,T]),$} \\
    \left(\phi(\Delta u(0)),-\phi(\Delta u(T))\right)\in\gamma(u(0),u(T+1)),
  \end{array}
  \right.
\end{equation}
under the assumption that $f$ satisfies hypothesis ($H_{f}$).
\medskip

According to Theorem \ref{thsolQ}, problem $[Q_\gamma(h)]$ has a unique solution for any $h\in X_T$. This enables us to define the {\it solution operator} $S_{\gamma}:X_T \to X_{T+2}$, by setting
$$S_{\gamma}(h):= u_h - \mbox{the unique solution of $[Q_\gamma(h)]$}.$$

\begin{proposition}\label{Sgammacont}
The operator $S_{\gamma}$ is continuous.
\end{proposition}

\noindent\proof Let $\{h_k\} \subset X_T \ni h$ be with $h_k \to h$, as $k \to \infty$. From \eqref{est}, we have
\begin{equation*}\label{cint}
\|u_{h_k}-u_h\|_T \leq \sqrt{T}\, |h_k -h |_{\infty} \qquad (k\in \mathbb{N}),
\end{equation*}
which implies that $u_{h_k}(n) \to u_h(n)$, as $k \to \infty$, for all $n\in \mathbb{Z}[1,T]$. On the other hand, from
\begin{equation*}\label{cext}
-\Delta\left[\phi(\Delta u_{h_k}(0))\right]+\Delta\left[\phi(\Delta u_h(0))\right] +  u_{h_k}(1)- u_h(1) = h_k(1)-h(1)
\end{equation*}
it follows
\begin{equation*}\label{cext0}
\Delta\left[\phi(\Delta u_{h_k}(0))\right] \to \Delta\left[\phi(\Delta u_h(0))\right] , \quad \mbox{ as } k \to \infty.
\end{equation*}
This, together with $u_{h_k}(1)\to u_h(1)$ and $u_{h_k}(2)\to u_h(2)$, yield that $u_{h_k}(0) \to u_{h}(0)$, as $k \to \infty$. A similar reasoning shows that $u_{h_k}(T+1) \to u_{h}(T+1)$, as $k \to \infty$.
Thus, $S_{\gamma}(h_{k}) \to S_{\gamma}(h)$ and the proof is complete. \cqfd

Next, we set
\begin{equation*}\label{defK}
    \mathcal{K}(a,\gamma):=\left\{w\in X_{T+2} \, : \, \|\Delta w\|_\infty<a,\ (w(0),w(T+1))\in D(\gamma)\right\}
\end{equation*}
and define
\begin{equation}\label{lambda1}
  \lambda_1 = \lambda_1(a,\gamma):=\inf\left\{\frac{\displaystyle\sum_{n=1}^{T+1}|\Delta w(n-1)|^2}{\|w\|_T^2} \, : \, w\in \mathcal{K}(a,\gamma),\; w|_{\mathbb{Z}[1,T]}\neq 0_{X_T}\right\}.
\end{equation}

\noindent This first eigenvalue-like constant satisfies $0\leq\lambda_1\leq 2$. Indeed, this is easy to see by noting that $w^\# \in X_{T+2}$ given by
$$w^\#(n):=\left\{
                               \begin{array}{lll}
                                 0_{\mathbb{R}^N} & \hbox{if $\ n=0, 2,\ldots, T+1;$} \\
                                 \;\\
                                 \left(\displaystyle{a}/{2}, \; 0_{\mathbb{R}^{N-1}}\right) & \hbox{if $\ n=1$}
                               \end{array}
                             \right.$$
fulfills $w^\# \in \mathcal{K}(a,\gamma)$, $w^\#|_{\mathbb{Z}[1,T]}\neq 0_{X_T}$.

\begin{proposition}\label{Kmarg}
If $\lambda_1>0$, then $\mathcal{K}(a,\gamma)$ is bounded in $X_{T+2}$.
\end{proposition}

\noindent\proof As $\lambda_1>0$, one has that $$u\mapsto\left(\sum_{n=1}^{T+1}|\Delta u(n-1)|^2+\lambda_1\|u\|_T^2\right)^{1/2}=:\|u\|_{\lambda_1}$$ is a norm on the finite dimensional space $X_{T+2}$ and using the definition of $\lambda_1$, we have that
$\lambda_1\|w\|_T^2<(T+1)\, a^2$ for all $w\in \mathcal{K}(a,\gamma)$. This implies that all $w\in \mathcal{K}(a,\gamma)$ satisfy $\|w\|_{\lambda_1}\leq a\sqrt{2(T+1)}$.\cqfd

\medskip

We introduce the mapping  $N_{f}:X_{T+2}\to X_T$ by
$$N_{f}(u)(n)=f(n,u(0),u(1),\ldots,u(T+1))+u(n),\quad  (u\in X_{T+2}).$$
It is easy to check that $N_{f}$ is continuous and, by virtue of Proposition \ref{Sgammacont}, one has that the operator
\begin{equation}\label{defScal}
    \mathcal{S}:=S_\gamma\circ N_f:X_{T+2}\to X_{T+2}
\end{equation}
is continuous. Also, it is straightforward to see that $u\in X_{T+2}$ is a solution of problem \eqref{bvpf} iff it is a fixed point of $\mathcal{S}$. Thus, to provide sufficient conditions for the solvability of problem   \eqref{bvpf}, we can apply the a priori estimates method, which in our case can be stated as follows.

\begin{proposition}\label{AEs} If the set
\begin{equation}\label{Adef}
    \mathcal{A}:=\{u\in X_{T+2}\; :\; \exists\ \mu\in(0,1]\mbox{ such that } u=\mu\mathcal{S}(u)\}
\end{equation}
is bounded in $(X_{T+2},\|\cdot\|_{T+2})$, then problem \eqref{bvpf} has at least one solution.
\end{proposition}

\noindent\proof This is immediate from Schaefer's theorem.\cqfd

\begin{theorem}\label{thf}
If $\lambda_1>0$, then problem \eqref{bvpf} has at least one solution.
\end{theorem}

\noindent\proof Let $u\in X_{T+2}$ be such that $\mathcal{S}(u)=\mu^{-1}u$, with some $\mu\in(0,1]$. From \eqref{defScal} and definition of the operator $S_{\gamma}$, in particular this means that
 $\|\Delta\mu^{-1}u\|_\infty<a$, $(\mu^{-1}u(0), \mu^{-1}u(T+1))\in D(\gamma)$, i.e. $\mu^{-1}u\in\mathcal{K}(a,\gamma)$. From Proposition \ref{Kmarg} we deduce that the set $\mathcal{A}$ in \eqref{Adef} is bounded and the conclusion follows by Proposition \ref{AEs}.\cqfd

Let $cone\, D(\gamma)$ be the \textit{conical hull} of the set $D(\gamma)$, that is
$$cone\, D(\gamma) := \{ \alpha z \, : \, \alpha \geq 0 , \, z\in D( \gamma )\}.$$

\begin{remark}\label{remcondsuf}
{\em A sufficient condition for having that $\lambda_1>0$ is
\begin{equation}\label{conDgamma}
\overline{cone \,D(\gamma)} \cap d_N^1 = \{ 0_{\mathbb{R}^N \times \mathbb{R}^N} \}.
\end{equation}
Indeed, setting
${\mathcal K}^\#:= \{w\in X_{T+2} \, : \, (w(0), w(T+1))\in \overline{cone\, D(\gamma)}\}$ and
$$\lambda_1^\#:=\inf\left\{\frac{\displaystyle\sum_{n=1}^{T+1}|\Delta w(n-1)|^2}{\|w\|_{T+2}^2} \, : \,  w\in {\mathcal K}^\#(\gamma)\setminus \{ 0_{X_{T+2}}\} \right \},$$
by virtue of \cite[Corollary 2.2]{JePr} and \eqref{conDgamma} it is not difficult to see that $\lambda_1^\#>0$ and the conclusion follows from $\lambda_1\geq \lambda_1^\#.$}
\end{remark}

An immediate consequence of Theorem \ref{thf} and Remark \ref{remcondsuf} is the following

\begin{corollary}\label{tcorf} If \eqref{conDgamma} holds true, then
\eqref{bvpf} has at least one solution.
\end{corollary}

\begin{example}\label{exantip}
\emph{For any $\alpha,\beta\in \mathbb{R}$, $\alpha \neq \beta$ and  $f$ satisfying hypothesis ($H_{f}$), problem
\begin{equation}\label{exalphab}
  \left\{
  \begin{array}{lcl}
    -\Delta\left[\phi(\Delta u(n-1))\right]=f(n, u(0), \dots , u(T+1)) & \hbox{$\quad (n \in \mathbb{Z}[1,T]),$} \\
    \alpha u(0)=\beta u(T+1), \; \beta \phi( \Delta u (0))=\alpha \phi(\Delta u(T))
  \end{array}
  \right.
\end{equation}
has at least one solution. To see this, consider the subspace
$Y=\{ (x,y) \in \mathbb{R}^N \times \mathbb{R}^N \, : \, \alpha x=\beta y \}$
and take $\gamma := \partial I_Y$. We know that
$$\partial I_Y(z)=N_Y(z)=Y^{\perp}=\{ (x,y) \in \mathbb{R}^N \times \mathbb{R}^N \, : \, \beta x=-\alpha y \} \quad (z\in Y),$$
and Corollary \ref{tcorf} applies.}
\end{example}

\medskip
\begin{remark}\label{remunivres}
{\em Existence of a solution under solely assumption $\lambda_1>0$  is a "universal" existence
result in the sense that no additional assumptions on $f$ in problem \eqref{bvpf} are needed. In this idea, the reader will notice that for any $f$ satisfying hypothesis $(H_{f})$,  system \eqref{sysgen} always has at least one solution that satisfies one of the classical boundary conditions: Dirichlet \eqref{dirhdis}, mixed Dirichlet-Neumann \eqref{mixedis}, antiperiodic \eqref{aperidis} or rotating periodic \eqref{Qperidis} -- if ker$(\mathbb{I}-\mathbb{U})=\{ 0_{\mathbb{R}^N} \}$.}
\end{remark}

\smallskip

In the case of Neumann boundary conditions \eqref{neuhdis}, periodic \eqref{peridis} or rotating periodic \eqref{Qperidis} -- with ker$(\mathbb{I}-\mathbb{U})\neq\{ 0_{\mathbb{R}^N} \}$, it is easy to see that $\lambda_1=0$, so Theorem \ref{thf} as well as Corollary
\ref{tcorf} are inoperative. In what follows, we aim to formulate conditions on $f$ that ensure the solvability of problem \eqref{bvpf}, without requiring the evaluation of the positivity of $\lambda_1.$

\smallskip

\begin{theorem}\label{thf1}
Assume that there is a real matrix $\displaystyle (c_{ij})_{i,j\in\mathbb{Z}[1,T]}$ with $\displaystyle\sum_{i=1}^T c_{ij}<1$ for all $j\in\mathbb{Z}[1,T]$ and a constant $c\in\mathbb{R}$, such that
\begin{equation}\label{condf}
    \langle f(n,x^0,\ldots,x^{T+1})|x^n\rangle\leq(c_{nn}-1)|x^n|^2+\sum_{j=1,j\neq n}^T c_{nj}|x^j|^2+c,
\end{equation}
for all $n\in\mathbb{Z}[1,T]$ and $(x^0,\ldots,x^{T+1})\in(\mathbb{R}^N)^{T+2}$.
Then, problem \eqref{bvpf} has at least one solution.
\end{theorem}

\noindent\proof According to Proposition \ref{AEs}, it suffices to show that the set $\mathcal{A}$ in \eqref{Adef} is bounded in $(X_{T+2},\|\cdot\|_{T+2})$. Thus, let $u\in X_{T+2}$ be such that $\mathcal{S}(u)=\mu^{-1}u$, with some $\mu\in(0,1]$. From \eqref{defScal} and definition of operators $S_{\gamma}$ and $N_f$, one has that $\mu^{-1}u$ is solution for problem $[Q_\gamma(N_f(u))]$ and, as $0_{\mathbb{R}^N \times \mathbb{R}^N}\in \gamma(0_{\mathbb{R}^N \times \mathbb{R}^N})$, from \eqref{est1} we have that
$$0 \leq {\mathcal M}(\mu^{-1}u,0_{X_{T+2}})+\omega(\mu^{-1}u,0_{X_{T+2}})+\|\mu^{-1}u\|_T^2=\sum_{n=1}^T\langle f(n,u(0),\ldots,u(T+1))+u(n)|\mu^{-1}u(n)\rangle.$$
Then, since ${\mathcal M}(\mu^{-1}u,0_{X_{T+2}})\geq0$, $\omega(\mu^{-1}u,0_{X_{T+2}})\geq0$, using \eqref{condf}, we obtain
\begin{eqnarray*}
  \mu^{-1}\|u\|_T^2 &\leq& \sum_{n=1}^T\left((c_{nn}-1)|u(n)|^2+\sum_{j=1,j\neq n}^T c_{nj}|u(j)|^2+c\right)+\|u\|_T^2 \\
  &=& \sum_{j=1}^T\sum_{n=1}^T c_{nj}|u(j)|^2+c \, T,
\end{eqnarray*}
which, together with $d:=\displaystyle\max_{j\in\mathbb{Z}[1,T]}\sum_{n=1}^T c_{nj}<1$ and $\mu\in(0,1]$, yield
$$\|u\|_T^2\leq d\|u\|_T^2+c\, T,$$
i.e., $\|u\|_T\leq r:=\sqrt{{c \,  T}/{(1-d)}}$.
On the other hand, as $\|\Delta\mu^{-1}u\|_\infty<a$, from \eqref{estimarew} we get that $|u(0)|\leq r/\sqrt{T}+Ta$, respectively $|u(T+1)|\leq r/\sqrt{T}+Ta$. Consequently, $\|u\|_{T+2}\leq\sqrt{2\left(r/\sqrt{T}+Ta\right)^2+r^2}$ and the proof is complete.
\cqfd

\begin{remark}
{\em The conclusion of Theorem \ref{thf1} still remains valid provided that the matrix $(c_{ij})_{i,j\in\mathbb{Z}[1,T]}$ is with $\displaystyle\sum_{j=1}^T c_{ij}<1$ for all $i\in\mathbb{Z}[1,T]$ and, instead of \eqref{condf}, we assume that
\begin{equation*}\label{condf1}
    \langle f(n,x^0,\ldots,x^{T+1})|x^n\rangle\leq(c_{nn}-1)|x^n|^2+\sum_{i=1,i\neq n}^T c_{in}|x^i|^2+c,
\end{equation*}
for all $n\in\mathbb{Z}[1,T]$ and $(x^0,\ldots,x^{T+1})\in(\mathbb{R}^N)^{T+2}$.}
\end{remark}
\medskip

\begin{example}
\emph{Let $(a_{ij})_{i,j\in\mathbb{Z}[1,T]}$ be a real matrix with $\displaystyle\sum_{i=1}^T |a_{ij}|<1$ for all $j\in\mathbb{Z}[1,T]$ and  $h\in X_T$. Then, problem
\begin{equation*}
  \left\{
  \begin{array}{lcl}
    -\Delta\left[\phi(\Delta u(n-1))\right]=\left(a_{nn}-1+\displaystyle\frac{\displaystyle\sum_{j=1,j\neq n}^T a_{nj}|u(j)|^2}{1+|u(n)|^2}\right)u(n)+\displaystyle\frac{1}{1+|u(n)|+|u(T+1)-u(0)|}\;h(n)\\ \hfill\hbox{$\ \ (n \in \mathbb{Z}[1,T]),$} \\
    \left(\phi(\Delta u(0)),-\phi(\Delta u(T))\right)\in\gamma(u(0),u(T+1)),
  \end{array}
  \right.
\end{equation*}
has at least one solution. To see this, Theorem \ref{thf1} applies with
$$f(n,x^0,\ldots,x^{T+1})=\left(a_{nn}-1+\displaystyle\frac{\displaystyle\sum_{j=1,j\neq n}^T a_{nj}|x^j|^2}{1+|x^n|^2}\right)x^n+\displaystyle\frac{1}{1+|x^n|+|x^{T+1}-x^0|}\;h(n) \quad (n \in \mathbb{Z}[1,T]),$$
$c_{ij}=|a_{ij}|$ ($i,j \in\mathbb{Z}[1,T]$) and $c=\displaystyle \max_{n\in \mathbb{Z}[1,T]}|h(n)|.$}
\end{example}
\medskip

\begin{theorem}\label{deltau}
Assume that, for all $n\in\mathbb{Z}[1,T]$, the following hold:
\begin{equation}\label{Mnrho}
M_n(\rho):=\sup\left\{|f(n,x^0,\ldots,x^{T+1})|:x^0,\ldots,x^{T+1}\in\mathbb{R}^N\mbox{ and } |x^n|\leq\rho\right\}<+\infty , \quad \mbox{for all } \rho>0
\end{equation}
and
\begin{equation}\label{fapriori}
 \limsup_{|x^n|\to\infty}\frac{\langle f(n,x^0,\ldots,x^{T+1}) | x^n\rangle}{|x^n|^2}<0,\quad \mbox{uniformly with } x^0,\ldots,x^{n-1},x^{n+1},\ldots,x^{T+1}\in\mathbb{R}^N.
\end{equation}

Then problem \eqref{bvpf} has at least one solution.
\end{theorem}
\noindent\proof Let $n\in\mathbb{Z}[1,T]$ be arbitrarily chosen. From \eqref{fapriori} we can find constants $\sigma(n)>0$ and $\rho(n)>0$ such that
\begin{equation*}
   \langle f(n,x^0,\ldots,x^{T+1}) | x^n\rangle<-\sigma(n)|x^n|^2,\quad \forall\ x^0,\ldots,x^{n-1},x^{n+1},\ldots,x^{T+1}\in\mathbb{R}^N \mbox{ and } x^n\in\mathbb{R}^N \mbox{ with } |x^n|>\rho(n).
\end{equation*}
On the other hand, \eqref{Mnrho} implies
$$\sup\left\{\langle f(n,x^0,\ldots,x^{T+1}) | x^n\rangle:x^0,\ldots,x^{T+1}\in\mathbb{R}^N\mbox{ and } |x^n|\leq\rho(n)\right\}\leq \rho(n) \, M_n(\rho(n))=:M^1(\rho(n))$$ and we can estimate
\begin{equation*}
   \langle f(n,x^0,\ldots,x^{T+1}) | x^n\rangle\leq-\sigma(n)|x^n|^2+\sigma(n)\rho(n)^2+M^1(\rho(n)),\quad \mbox{for all }\ x^0,\ldots,x^{T+1}\in\mathbb{R}^N.
\end{equation*}
Then, Theorem \ref{thf1} applies with $c_{nn}=1-\sigma(n)$ for $n\in \mathbb{Z}[1,T]$, $c_{ij}=0$ for $i,j\in \mathbb{Z}[1,T]$, $i\neq j$ and $c=\max_{n\in \mathbb{Z}[1,T]}\{\sigma(n)\rho(n)^2+ M^1(\rho(n))\}$. \cqfd
\medskip

The following is an extended discrete variant of Corollary 3.2 in \cite{JeSe}.
\smallskip

\begin{corollary}\label{cdeltau} Let $g:\mathbb{Z}[1,T] \times (\mathbb{R}^N)^{T+2} \to \mathbb{R}^N$ be continuous and bounded.
If the continuous function $\ell:\mathbb{Z}[1,T]\times\mathbb{R}^N\to\mathbb{R}^N$ satisfies
\begin{equation}\label{lsat}
\limsup_{|x|\to\infty}\frac{\langle \ell(n,x) | x\rangle}{|x|^2}<0 \qquad (n\in\mathbb{Z}[1,T]),
\end{equation}
then problem
\begin{equation*}
  \left\{
  \begin{array}{lcl}
    -\Delta\left[\phi(\Delta u(n-1))\right]=\ell(n,u(n)) + g(n, u(0), \dots , u(T+1))  & \hbox{$\quad (n\in\mathbb{Z}[1,T]),$} \\
    \left(\phi(\Delta u(0)),-\phi(\Delta u(T))\right)\in\gamma(u(0),u(T+1))
  \end{array}
  \right.
\end{equation*}
has at least one solution.
\end{corollary}
\noindent\proof Theorem \ref{deltau} applies with
$$f(n, x^0, \dots , x^{T+1})= \ell (n, x^n)+g(n, x^0, \dots , x^{T+1}) \quad (n\in \mathbb{Z}[1,T], \, x^0, \dots , x^{T+1}\in \mathbb{R}^N).$$
\cqfd


\begin{example}\label{exnNeSt}
\emph{Let $G_1,\, G_2:\mathbb{R}^N \to \mathbb{R}$ be convex. If $\nabla G_i$ exists and is continuous,
$\nabla G_i(0_{\mathbb{R}^N}) = 0_{\mathbb{R}^N}$
and  $G_i(0_{\mathbb{R}^N})=0$ ($i=1,2$), then, for any $h\in X_T$, system
$$ -\Delta\left[\phi(\Delta u(n-1))\right]+ u(n) =\displaystyle \frac{ u(n-1)}{1+|u(n-1)|} +h(n)   \hbox{$\quad (n\in\mathbb{Z}[1,T])$}$$
has at least one solution satisfying the Neumann-Steklov boundary conditions
\begin{equation}\label{nsbc}
\phi(\Delta u(0)) =\nabla G_1(u(0)), \quad \phi(\Delta u(T)) =-\nabla G_2(u(T+1)).
\end{equation}
Indeed, \eqref{nsbc} is of type \eqref{ncon} with $K=\mathbb{R}^N \times \mathbb{R}^N$ (hence, $N_K(z)=\{ 0_{\mathbb{R}^N \times \mathbb{R}^N} \}$ ($z\in \mathbb{R}^N \times \mathbb{R}^N$)) and $G(x,y)=G_1(x)+G_2(y)$ ($(x,y)\in \mathbb{R}^N \times \mathbb{R}^N$). Then, Corollary \ref{cdeltau} easily applies with $\ell (n,x)=-x$ ($n\in \mathbb{Z}[1,T]$, $x\in \mathbb{R}^N$) and $g(n, x^0, \dots , x^{T+1})= |x^{n-1}|/(1+|x^{n-1}|)+h(n)$ ($n\in \mathbb{Z}[1,T]$, $x^0, \dots , x^{T+1}\in \mathbb{R}^N$).
}
\end{example}

\medskip
\begin{corollary}\label{cdeltaunmm} Let $\mathbb{M}$ be a real $(2N \times 2N)$-matrix which is positive semi-definite (i.e. $\langle \! \langle \mathbb{M} z \, | \, z  \rangle \! \rangle \geq 0$, for all $z\in \mathbb{R}^{N} \times \mathbb{R}^N$). Then, for any $g:\mathbb{Z}[1,T] \times \mathbb{R}^N \times \mathbb{R}^N \to \mathbb{R}^N$ continuous and bounded and any continuous $\ell:\mathbb{Z}[1,T]\times\mathbb{R}^N\to\mathbb{R}^N$ satisfying \eqref{lsat}, problem
\begin{equation}\label{nmaxm}
  \left\{
  \begin{array}{lcl}
    -\Delta\left[\phi(\Delta u(n-1))\right] =\ell(n,u(n))+ g(n, u(n), \Delta u(n))  & \hbox{$\quad (n\in\mathbb{Z}[1,T]),$} \\
    \\
    \left(
\begin{array}{clrc}
  \phi(\Delta u(0)) \\[6pt]
  -\phi(\Delta u(T))
\end{array}
\right)=\mathbb{M}\
\left(
\begin{array}{clrc}
  u(0) \\[6pt]
  u(T+1)
\end{array}
\right)
  \end{array}
  \right.
\end{equation}
has at least one solution.
\end{corollary}
\noindent\proof As the linear operator
\begin{equation}\label{gammaa}
\gamma_\mathbb{M} (z) = \mathbb{M}\, z \qquad (z\in \mathbb{R}^{N} \times \mathbb{R}^N).
\end{equation}
is maximal monotone \cite[Example 1.5 (b)]{[Ph]}, the conclusion is immediate from Corollary \ref{cdeltau}.
\cqfd

\begin{remark}
{\em If $\mathbb{M}$ is not symmetric, then $\gamma_\mathbb{M}$ in \eqref{gammaa} will not be cyclically monotone (see p. 240 in \cite{Rock}) and thus, by a classical result on the potentiality of maximal monotone operators \cite{Rock2}, there is no lower semicontinuous proper convex function $j:\mathbb{R}^{N} \times \mathbb{R}^N \to (-\infty , +\infty]$ such that $\gamma_\mathbb{M} = \partial j$. Therefore, if $\mathbb{M}$ is positive semi-definite, but it is not symmetric, then the boundary condition in problem \eqref{nmaxm} is of type \eqref{bcsysgen} with $\gamma$ which is not the subdifferential of a proper lower semicontinuous convex function $j$.}
\end{remark}

\begin{example}\label{exnonpot}
\emph{For any $h\in X_T$  and $\varepsilon >0$, $p\geq 2$, problem
\begin{equation*}
  \left\{
  \begin{array}{lcl}
    -\Delta\left[\phi(\Delta u(n-1))\right]+\varepsilon |u(n)|^{p-2} u(n) =\displaystyle \frac{ \Delta u(n)}{\sqrt{1+|\Delta u(n)|^2}} +h(n)  & \hbox{$\quad (n\in\mathbb{Z}[1,T]),$} \\
    \\
    \phi(\Delta u(0))=u(T+1), \, \phi(\Delta u(T))=u(0)
  \end{array}
  \right.
\end{equation*}
has at least one solution. To see this, Corollary \ref{cdeltaunmm} applies with
$$\mathbb{M}= \left[
\begin{array}{clrc}
\mathbb{O} & \mathbb{I} \\
-\mathbb{I} & \mathbb{O}
\end{array}\right] , \quad \ell(n,x)=-\varepsilon |x|^{p-2}x \quad \! \! \mbox{ and }\! \! \quad g(n,x,y)=\displaystyle \frac{y}{\sqrt{1+|y|^2}}+h(n) \quad (n\in \mathbb{Z}[1,T], \, x,y\in \mathbb{R}^N).$$
}
\end{example}
\bigskip

\section{Variational solutions}

In this section, besides hypothesis ($H_{\Phi}$), we assume ($H_{F}$), ($H_{j}$)  and we are concerned with existence of solutions for the potential system \eqref{potsysgendis} submitted to the potential boundary condition \eqref{potbcsysgen}, that is the problem
\begin{equation}\label{bvppot}
  \left\{
  \begin{array}{lcl}
    -\Delta\left[\phi(\Delta u(n-1))\right]=\nabla_u F(n, u(n)) \quad (n \in \mathbb{Z}[1,T]) \\
    \left(\phi(\Delta u(0)),-\phi(\Delta u(T))\right)\in\partial j (u(0),u(T+1)).
  \end{array}
  \right.
\end{equation}
The solutions of such a problem occurs as critical points of a certain energy function. In this view, notice first that, as $\partial j$ is a maximal monotone operator and $0_{\mathbb{R}^N \times \mathbb{R}^N}\in \partial j (0_{\mathbb{R}^N \times \mathbb{R}^N})$, problem
\begin{equation*}
  [Q_{\partial j}(h)]\qquad\qquad
  \left\{
  \begin{array}{lcl}
    -\Delta\left[\phi(\Delta u(n-1))\right]+u(n)=h(n) & \hbox{$\quad (n \in \mathbb{Z}[1,T]),$} \\
    \left(\phi(\Delta u(0)),-\phi(\Delta u(T))\right)\in \partial j (u(0),u(T+1)),
  \end{array}
  \right.
  \end{equation*}
has a unique solution $u_h\in X_{T+2}$, for any $h\in X_T$ (see Theorem \ref{thsolQ}).
We introduce:

\begin{equation*}
\mathbb{K}(a):=\{u\in X_{T+2}:\ \|\Delta u\|_\infty\leq a\},
\end{equation*}

\begin{equation*}
\Psi: X_{T+2} \to (-\infty, +\infty], \qquad
    \Psi (u)=\left\{
\begin{array}{ll} \displaystyle \sum_{n=1}^{T+1}\Phi[\Delta u(n-1)],\quad \mbox { if }  u\in \mathbb{K}(a),\\
\cr +\infty ,\quad \mbox { if } u \in X_{T+2} \setminus \mathbb{K}(a),
\end{array} \right.
\end{equation*}

\begin{equation*}
J:X_{T+2} \to (-\infty, +\infty], \qquad J(u)= \left (u(0), u(T+1) \right ),
\end{equation*}

 \begin{equation*}
\mathcal{\textbf{F}}:X_{T+2}\to \mathbb{R}, \qquad    \mathcal{\textbf{F}}(u)=-\sum_{n=1}^T F(n,u(n)).
\end{equation*}
It is a standard matter to check that: $\mathbb{K}(a)$ ($=D(\Psi)$) is closed and convex,  $\Psi$ and $J$ are proper, convex and lower semicontinuos, $\Psi+J$ is bounded from below and $\textbf{F}$ is of class $C^1$ on $X_{T+2}$,
its derivative being given by
\begin{equation*}
\langle \mathcal{\textbf{F}}'(u),v \rangle=-\sum_{n=1}^T \langle \nabla _u F(n,u(n)) \, | \, v(n)\rangle \qquad (u,v \in  X_{T+2}).
\end{equation*}
Also, it is clear that $D(J)= \left \{ u \in X_{T+2} \, : \, \left (u(0), u(T+1) \right ) \in D(j) \right \}$.
Then, { \it the energy functional associated to} \eqref{bvppot} will be
\begin{equation*}\label{Erond}
\mathcal{E}:X_{T+2} \to (-\infty, +\infty], \qquad \mathcal{E}:=\Psi +J + \mathcal{\textbf{F}},
\end{equation*}
which fits the structure required by Szulkin's critical point theory \cite{Sz}. Saying that $\mathcal{E}$ is the energy functional of \eqref{bvppot}, we have to prove that a critical point of $\mathcal{E}$ in the sense of this theory is indeed a solution of problem \eqref{bvppot}. We will do this in Theorem \ref{cpsol} below. For the convenience of the reader, we recall some basic notions of this theory in the context of our framework. Thus, {\it a critical point of } $\mathcal{E}$ means
an element $w\in D(\Psi +J)$ which satisfies
\begin{equation}\label{critpointdef}
(\Psi + J)(v)-(\Psi + J)(w) + \langle \mathcal{\textbf{F}}^{\prime} (w),v-w \rangle \geq 0, \quad \mbox{ for all } v \in D(\Psi + J).
\end{equation}
A sequence $\{ w_k \} \subset D(\Psi + J)$ will be called a \textit{(PS)-sequence} for $\mathcal{E}$ if ${\mathcal E}(w_k) \to c \in \mathbb{R}$ and
\begin{equation}\label{PSseq}
(\Psi + J)(v)-(\Psi + J)(w_k)+\langle \mathcal{\textbf{F}}^{\prime} (w_k),v-w_k \rangle \geq -\varepsilon_k\| v-w_k \|_{T+2}, \quad \mbox{ for all } v \in D(\Psi + J)
\end{equation}
\noindent where $\varepsilon_k \to 0$. The functional ${\mathcal E}$ is said \textit{to satisfy the (PS)-condition} if any (PS)-sequence has a convergent subsequence in $X_{T+2}$.

\begin{proposition}\label{pivar}
The solution $u_h$ of problem $[Q_{\partial j}(h)]$ is the unique solution in $D(\Psi+J)$ of the variational inequality
\begin{equation}\label{iv}
\Psi(v)+J(v)-\Psi(u)-J(u) +\sum_{n=1}^{T}\langle u(n)-h(n) | v(n)-u(n) \rangle \geq 0 , \quad \mbox{ for all } v \in D(\Psi + J)
\end{equation}
and the unique minimum point of the strictly convex function $E:X_{T+2}\to (-\infty,+\infty]$ defined by
$$E(v)=\Psi(v)+J(v)+\frac{1}{2}\|v\|_T^2- \sum_{n=1}^{T}\langle h(n)|v(n)\rangle \quad (v\in X_{T+2}).$$
\end{proposition}

\noindent\proof
Let $v\in \mathbb{K}(a)$ be with $(v(0),v(T+1))\in D(j)$ be arbitrarily chosen. For all $(x,y)\in\mathbb{R}^N\times \mathbb{R}^N$ we have
\begin{equation}\label{estimj}
j(x,y)-j(u_h(0),u_h(T+1))\geq \langle\phi(\Delta u_h(0))|x-u_h(0)\rangle-\langle \phi(\Delta u_h(T))|y-u_h(T+1) \rangle.
\end{equation}
\noindent Then, using the convexity of $\Phi$, estimation \eqref{estimj} and summation by parts formula \eqref{sbpf}, we get
\begin{equation*}
\begin{split}
 \Psi(v)+J(v)-\Psi(u_h)-J(u_h)=&\sum_{n=1}^{T+1}\{\Phi[\Delta v(n-1)]-\Phi[\Delta u_h(n-1)]\}+J(v)-J(u_h)\\
\geq &\sum_{n=1}^{T+1} \langle \phi(\Delta u_h(n-1))|\Delta (v- u_h)(n-1)\rangle\\
&+\langle\phi(\Delta u_h(0))| (v-u_h)(0)\rangle - \langle\phi(\Delta u_h(T))| (v-u_h)(T+1) \rangle\\
=& \sum_{n=1}^{T}\langle -\Delta(\phi(\Delta u_h(n-1)))| (v-u_h)(n)\rangle.
\end{split}
\end{equation*}
Since $u_h$ is the solution of problem $[Q_{\partial j}(h)]$, we obtain that $u_h$ verifies \eqref{iv}.

\smallskip
Next, using the elementary inequality
$$\frac{|y|^2}{2}-\frac{|x|^2}{2}\geq \langle x|y-x\rangle, \quad (x,y\in\mathbb{R}^N),$$
it is easy to check that any solution of \eqref{iv} is a minimum point of $E$.
Further, observe that to prove that $E$ is strictly convex it suffices to show that the mapping
$$\pi :\mathbb{K}(a) \to \mathbb{R}, \qquad \pi(v)= \Psi(v)+\frac{1}{2}\|v\|_T^2\quad \left ( = \sum_{n=1}^{T+1}\Phi[\Delta u(n-1)]+\frac{1}{2}\|v\|_T^2\right )$$
is strictly convex. Thus, let $u,v\in \mathbb{K}(a)$, $u\neq v$ and $\alpha \in (0,1)$. We claim that
\begin{equation}\label{strcE}
\pi( \alpha u +(1- \alpha )v) <\alpha \pi (u) +(1- \alpha ) \pi (v).
\end{equation}
Indeed, if $u|_{\mathbb{Z}[1,T]} \neq v|_{\mathbb{Z}[1,T]}$, this follows by the strict convexity of the mapping $\| \cdot \|_T^2$. In the remaining case $u|_{\mathbb{Z}[1,T]} = v|_{\mathbb{Z}[1,T]}$, we must have either
 $u(0) \neq v(0)$ or $u(T+1) \neq v(T+1)$. Let us say that $u(0) \neq v(0)$ (similar reasoning when $u(T+1) \neq v(T+1)$). Then, clearly one has that $\Delta u(0) \neq \Delta  v(0)$ and by the strict convexity of $\Phi$ we have
 $$\Phi(\alpha \Delta u(0) +(1- \alpha )\Delta v(0))< \alpha \Phi (\Delta u(0)) +(1- \alpha) \Phi (\Delta v(0)).$$
This yields
\begin{equation*}
\Psi( \alpha u +(1- \alpha) v) <\alpha \Psi (u) +(1- \alpha ) \Psi(v),
\end{equation*}
which implies \eqref{strcE}.
\cqfd
\begin{theorem}\label{cpsol}
Any critical point of $\mathcal{E}$ is a solution of problem \eqref{bvppot}.
\end{theorem}

\noindent \proof
Let $w\in D(\Psi+J)$ be a critical point of $\mathcal{E}$. This means that it satisfies
$$\Psi(v)+J(v)-\Psi(w)-J(w)+\sum_{n=1}^{T}\langle w(n)-[w(n)+\nabla _uF(n,w(n))]| v(n)-w(n) \rangle\geq 0, \quad \mbox{for all } v\in D(\Psi+J)),$$
which is a variational inequality of type \eqref{iv} with $h=w+\nabla_u F(\cdot,w)$. From Proposition \ref{pivar} we have that $w=u_{w+\nabla_u F(\cdot,w)}$ is a solution of  $[Q_{\partial j}(w+\nabla_u F(\cdot ,w))]$, thus it solves problem \eqref{bvppot}.
\cqfd

\medskip

\bigskip

\section{Minimum energy and saddle-point solutions}

In what follows we keep hypotheses ($H_{\Phi}$), ($H_{F}$) and ($H_{j}$) and we deal with existence of solutions of problem \eqref{bvppot} that appear either as minimum points of the energy functional ${\mathcal E}$ or as saddle-points of ${\mathcal E}$. In this view, first we need a result expresing the constant $\lambda_1$ in terms of the domain $D(\Psi + J)=D({\mathcal E})$, which will allow us to state a "universal"  existence result  on minimum energy solutions (see Theorem \ref{l1poz} $(iii)$, below).

\begin{proposition}\label{Kmarg1} Let $\lambda_1=\lambda_1(a, \partial j)$ be the constant defined by \eqref{lambda1}, with $\gamma=\partial j$. Then it holds
\begin{equation}\label{lambda11}
   \lambda_1=\inf\left\{\frac{\displaystyle\sum_{n=1}^{T+1}|\Delta w(n-1)|^2}{\|w\|_T^2} \, : \, w\in D(\Psi +J),\; w|_{\mathbb{Z}[1,T]}\neq 0_{X_T}\right\}.
\end{equation}
\end{proposition}

\noindent \proof Denote by $R$ the Rayleigh quotient entering in \eqref{lambda11}, that is
$$R(w):={\displaystyle \left (\sum_{n=1}^{T+1}|\Delta w(n-1)|^2 \right )} /{\|w\|_T^2} \quad (w\in X_{T+2}, \, \| w \| _{T}>0).$$
It is easy to see that the mapping $R:\left \{w\in X_{T+2} \, : \, \| w \| _{T}>0 \right \} \to \mathbb{R}$ is continuous. Next, set
\begin{equation}\label{eicL}
\mathcal{L}(a , j):= \left \{w \in X_{T+2} \, : \, \| \Delta w \|_{\infty }< a, \, (w(0),w(T+1))\in D(j) \right \}
\end{equation}
and define
\begin{equation*}
\beta:=  \inf\left\{ R(w) \, : \, w\in \mathcal{L}(a , j),\; \|w\|_{T} > 0 \right\},
\end{equation*}
\begin{equation*}
\alpha:= \inf\left\{ R(w) \, : \, w\in D(\Psi +J),\; \|w\|_{T} > 0\right\}.
\end{equation*}
Note that the following equality holds
\begin{equation}\label{beeqal}
 \beta = \alpha.
 \end{equation} Indeed, since $D(\Psi +J) \supset \mathcal{L}(a, j),$
we only have to check that
\begin{equation}\label{blal}
 \beta \leq \alpha.
 \end{equation} For this, let $\varepsilon >0$ be arbitrarily chosen and $w_{\varepsilon} \in X_{T+2}$ be with
$$ \|\Delta  w_{\varepsilon} \|_{\infty} \leq a,\quad \left ( w_{\varepsilon}(0), w_{\varepsilon}(T+1)\right ) \in D(j), \quad \| w_{\varepsilon} \| _{T}>0 \, \, \mbox{ and } \, \, R(w_{\varepsilon}) <\alpha +\varepsilon $$
(such a $w_{\varepsilon}$ is known to exist by the definition of $\alpha$). Pick $\delta \in (0,1)$ and observe that, as $0_{\mathbb{R}^N \times \mathbb{R}^N}\in D(j)$ and $D(j)$ is a convex set, the function $\delta  w_{\varepsilon}$ belongs to $\mathcal{L}(a , j)$ and $\beta \leq R(\delta  w_{\varepsilon})=R( w_{\varepsilon})<\alpha +\varepsilon$, which yields \eqref{blal} by letting $\varepsilon \to 0_+.$
\medskip

Now, on account of \eqref{beeqal} and since
$\mathcal{L}(a , j) \supset \mathcal{K}(a, \partial j),$
to complete the proof, it remains to show that
\begin{equation}\label{lambda112}
\lambda_1 \leq \beta.
\end{equation}
In this respect, we proceed in two steps. {\sl Step 1.} Let $w\in \mathcal{L}(a , j)$ be with $\|w\|_{T} > 0$. We {\sl claim} that there is a sequence $\{ w_l\} \subset \mathcal{K}(a, \partial j)$ with $\| w_l \|_T>0$ for all $l\in \mathbb{N}$, such that $w_l \to w$ in $X_{T+2}$, as $l \to \infty$. To prove this we use that $D(\partial j)$ is dense in $D(j)$ \cite[Proposition 2.11]{Br}. Thus, we can find a sequence $\{ (x^l, y^l) \} \subset D(\partial j)$ with $(x^l, y^l) \to (w(0),w(T+1))$ in $\mathbb{R}^N \times \mathbb{R}^N$, as $l \to \infty$. Let $w_l \in X_{T+2}$ be defined by
\begin{equation*}
 w_l(n)= \left\{
  \begin{array}{ll}
    x^l \quad & \mbox{ if }n=0; \\
    w(n)\quad & \mbox{ if }n\in \mathbb{Z}[1,T];\\
    y^l  \quad & \mbox{ if }n=T+1.
  \end{array}
  \right.
\end{equation*}
We have that $w_l \to w$ in $X_{T+2}$, as $l\to \infty$, and we can assume that $\|w_l\|_T >0$ for all $l\in \mathbb{N}$ - by the continuity of the seminorm $\| \cdot \|_T$ on $X_{T+2}$.
Since $\|\Delta w\|_{\infty} <a$ and $\Delta w_l(n-1) \to \Delta w(n-1)$  for all $n\in \mathbb{Z}[1,T+1]$, there is some $l_w\in \mathbb{N}$ such that
$\|\Delta w_l\|_{\infty} <a$ for all $l \geq l_w$. Then, clearly, we can assume that the whole sequence  $\{ w_l\}$ is in $\mathcal{K}(a, \partial j)$ and the claim is proved.

\medskip
{\sl Step 2.} Let $\varepsilon >0$ be arbitrarily chosen and $\{ v_k \} \subset \mathcal{L}(a , j)$ be with $\|v_k\|_T>0$ and
\begin{equation}\label{rey}
R (v_k) < \beta +\frac{1}{k} \qquad (k\in \mathbb{N}).
\end{equation}
 Let $k\in \mathbb{N}$ be such that $1/k < \varepsilon$. We apply Step 1 with $w=v_k$ and find a sequence $\{ w^k_l\}\subset \mathcal{K}(a, \partial j)$ with $\|w^k_l\|_T >0$ for all $l\in \mathbb{N}$, such that $w^k_l \to v_k$ in $X_{T+2}$, as $l \to \infty$. Thus, using \eqref{rey}, we get
$$ \lambda_1 \leq R (w^k_l) \stackrel{l \to \infty}{\rightarrow} R (v_k)< \beta + \varepsilon$$
which, letting $\varepsilon \to 0_+$, implies \eqref{lambda112} and the proof is complete.
\cqfd
\medskip

\begin{remark}\label{remcondsufj}
{\em ({\it i}) With $\mathcal{L}(a , j)$ defined in \eqref{eicL}, in the previous proof we obtained -- as a byproduct,  that besides \eqref{lambda11},  it holds
\begin{equation*}\label{lambda6}
   \lambda_1=\inf\left\{\frac{\displaystyle\sum_{n=1}^{T+1}|\Delta w(n-1)|^2}{\|w\|_T^2} \, : \, w\in \mathcal{L}(a , j),\; w|_{\mathbb{Z}[1,T]}\neq 0_{X_T}\right\}.
\end{equation*}
\medskip

 ({\it ii}) From Remark \ref{remcondsuf} we have that a sufficient condition for having that $\lambda_1>0$ is
\begin{equation}\label{conDgammaj}
\overline{cone \,D(\partial j)} \cap d_N^1 = \{ 0_{\mathbb{R}^N \times \mathbb{R}^N} \}.
\end{equation}
By the density of $D(\partial j)$ in $D(j)$ it is easy to check that $cone \,D( j) \subset \overline{cone \,D(\partial j)}$, which actually yields
\begin{equation*}\label{reyll}
\overline{cone \,D( j)} = \overline{cone \,D(\partial j)}
\end{equation*}
and thus \eqref{conDgammaj} is equivalent with
\begin{equation*}\label{conDgammajj}
\overline{cone \,D(j)} \cap d_N^1 = \{ 0_{\mathbb{R}^N \times \mathbb{R}^N} \}.
\end{equation*}}
\end{remark}

\begin{theorem}\label{l1poz} If $\lambda_1>0$ then the following hold:

\medskip
(i) for any $w\in  D(\Psi + J)$ one has
\begin{equation}\label{overl1d}
|w(m)| \leq a \left ( \sqrt{\frac{T+1}{T\lambda_1}} +T \right ) \qquad (m\in \mathbb{Z}[0, T+1]) ;
\end{equation}

\medskip
(ii) ${\mathcal E}$ is bounded from below and satisfies the (PS)-condition;

\medskip
(iii) ${\mathcal E}$ attains its infimum at some $u\in D(\Psi + J)$, which is a solution of problem \eqref{bvppot}.
\end{theorem}

\noindent \proof $(i)$ As  $\|\Delta w \|_{\infty} \leq a$, the estimation \eqref{estimarew} holds. Also, \eqref{lambda11} implies
\begin{equation}\label{rewxx}
\|w\|_T\leq a \sqrt{\frac{T+1}{ \lambda _1}}
\end{equation}
Then \eqref{overl1d} follows by \eqref{estimarew} and \eqref{rewxx}.

\medskip
$(ii)$  By \eqref{overl1d} the set $D(\Psi + J)$ is bounded, whence relatively compact in $X_{T+2}$. This clearly implies that ${\mathcal E}$ satisfies the (PS)-condition and that the continuous function $\mathcal{\textbf{F}}$ is bounded on $D(\Psi + J)$ -- which yields that ${\mathcal E}$ is bounded from below.

\medskip
$(iii)$  This is immediate from $(ii)$, Theorem 1.7 in \cite{Sz} and Theorem \ref{cpsol}.
\cqfd

\begin{remark}\label{solenmin}
{\em  The novelty in Theorem \ref{l1poz} $(iii)$ does not consist in the fact that problem \eqref{bvppot} has a solution, since this follows from the more general  Theorem \ref{thf}. What is new is that \eqref{bvppot} has a minimum energy solution. In this respect, it should be noticed that a system with a potential structure of  type \eqref{potsysgendis}, with $F$ satisfying $(H_F)$, always has a minimum energy solution provided that it is associated with one of the classical boundary conditions listed in Remark \ref{remunivres}. Also, related to Example \ref{exantip}, it is easy to see that if $\alpha \neq \beta$, then system \eqref{potsysgendis} has at least one minimum energy solution  satisfying the boundary condition from problem \eqref{exalphab}.
}
\end{remark}

Next, for $p,q\in\mathbb{Z}$ with $p<q$ and $v:\mathbb{Z}[p,q]\to \mathbb{R}^N$, we introduce the notations
$$\overline{v}:=\frac{1}{q-p+1}\sum_{k=p}^{q}v(k) \quad \text{and } \quad \widetilde{v}:=v-\overline{v}.$$
Using Lemma \ref{lemarmj} it is straightforward to see that, if $v\in X_{T+2}$ is with $\|\Delta v\|_{\infty}\leq a$ and $\overline{v}=0$, then $|v(m)|\leq (T+1) \, a$ for all $m\in\mathbb{Z}[0,T+1]$. It follows that any
$v\in X_{T+2}$ with $\|\Delta v\|_{\infty}\leq a$ satisfies
\begin{equation}\label{estvtilde}
| \widetilde{v}(m) | \leq (T+1) \, a \qquad (m \in \mathbb{Z}[0,T+1]).
\end{equation}

\noindent We introduce the set
$$D_\delta(\Psi+J):=\{v\in D(\Psi+J)\,  : \, |\overline{v}|\leq \delta\} \qquad (\delta >0).$$

\begin{theorem}\label{thmminimumenergy} Assume that there is some $\delta>0$ such that
\begin{equation}\label{relinf}
\inf_{D_\delta(\Psi+J)}\mathcal{E}=\inf_{D(\Psi+J)}\mathcal{E}.
\end{equation}
Then $\mathcal{E}$ is bounded from below on $X_{T+2}$ and attains its infimum at some $u\in D_\delta(\Psi+J)$, which is a solution of problem \eqref{bvppot}.
\end{theorem}

\noindent \proof From \eqref{estvtilde} we deduce that any $v \in D_\delta(\Psi+J)$ satisfies
$$\|v\|_{T+2}=(\delta+(T+1)\, a)\sqrt{T+2},$$
which show that $D_\delta(\Psi+J)$ is bounded in $X_{T+2}$. On the other hand, it is easy to verify that $D_\delta(\Psi+J)$ is a closed set, so $D_\delta(\Psi+J)$ is compact. Using that $\mathcal{E}$ is lower semicontinuous together with
\eqref{relinf}, we obtain that there is some $u\in D_\delta(\Psi+J)$ such that
$$\mathcal{E}(u)=\min_{D_\delta(\Psi+J)}\mathcal{E}=\inf_{X_{T+2}}\mathcal{E}.$$
Then, according to \cite[Proposition 1.1]{Sz}, $u$ is a critical point of $\mathcal{E}$, whence a solution of \eqref{bvppot} -- by virtue of Theorem \ref{cpsol}.
\cqfd

Instead of \eqref{bvppot}, let us consider a slightly modified version of it, namely the problem
\begin{equation}\label{forcterm}
\left\{
  \begin{array}{lcl}
    -\Delta\left[\phi(\Delta u(n-1))\right]=\nabla_u F(n, u(n)) +h(n) \quad (n \in \mathbb{Z}[1,T]) \\
    \left(\phi(\Delta u(0)),-\phi(\Delta u(T))\right)\in\partial j (u(0),u(T+1)),
  \end{array}
  \right.
  \end{equation}
where the forcing term $h$ belongs to $X_{T}$. Then, setting
$$F_h(n,u):=F(n,u)+\langle h(n)| u\rangle \quad ((n,u)\in\mathbb{Z}[1,T]\times \mathbb{R}^N),$$
it is straightforward to see that $F_h$ satisfies $(H_F)$,  that is $(H_F)$ with $F_h$ instead of $F$.
In this case, the energy functional $\mathcal{E}_{h}:X_{T+2}\to (-\infty,+\infty]$ attached to \eqref{forcterm} will given by
$$\mathcal{E}_h(u)=\Psi(u)+J(u)-\sum_{n=1}^{T}F(n,u(n))-\sum_{n=1}^{T}\langle h(n)|u(n) \rangle \quad (u\in X_{T+2})$$
and all the above in this section are valid with $\mathcal{E}_h$ instead of $\mathcal{E}$, when referring to \eqref{forcterm}.

\begin{theorem}\label{thperiodic}
Assume that:
\begin{enumerate}
\item[(i)] if $z\in D(j)$ then $j(z+\zeta)=j(z)$ for all $\zeta\in d_{N}^{1}$;
\item[(ii)] $F(n,u)$ is $\omega_i$-periodic ($\omega_i>0$) with respect to each $u_i$ ($i=\overline{1,N}$) for all $n\in\mathbb{Z}[1,T]$.
\end{enumerate}
Then, for any $h\in X_T$ with $\overline{h}=0_{\mathbb{R}^N}$, the functional $\mathcal{E}_h$ is bounded from below and attains its infimum at some $u\in D(\Psi + J)$, which is a solution of problem \eqref{forcterm}.
\end{theorem}

\noindent \proof
Set $\omega:=\max\{\omega_i:i=\overline{1,N}\}$. We show that
\begin{equation}\label{egmulte}
\{\mathcal{E}_h(u)\, : \, u\in D(\Psi+J)\}=\{\mathcal{E}_h(v)\, : \, v\in D_{\omega N}(\Psi+J)\}
\end{equation}
and the conclusion will follow by Theorem \ref{thmminimumenergy}. At its turn, \eqref{egmulte} reduces to prove that for each $u\in D(\Psi+J)$ there is some $v\in D_{\omega N}(\Psi+J)$ such that $\mathcal{E}_h(u)=\mathcal{E}_h(v)$. And indeed, let $u\in D(\Psi+J)$ and denote by  $e^1,\dots,e^N$ canonical basis in $\mathbb{R}^N$. We fix $k_i=k_i(\overline{u})\in\mathbb{Z}$ so that $\langle \overline{u} \, | \, e^i \rangle-k_i\omega_i\in [0,\omega_i)$ ($i=\overline{1,N}$); such a family of $k_i$ exists and is unique. Then, using assumptions $(i)$, $(ii)$ and that $\overline{h}=0$, we obtain
\begin{equation}\label{demcor}
\mathcal{E}_h(u)=\mathcal{E}_h\left(u-\sum_{i=1}^{N}k_i\omega_ie^i\right).
\end{equation}
Denoting $\overline{u}^0:=\displaystyle \sum_{i=1}^{N}\left (\langle \overline{u}\, | \, e^i \rangle-k_i\omega_i \right )e^i$, we define $v:=\overline{u}^0+\widetilde{u}$.
As $\overline{v}=\overline{u}^0$ (whence $|\overline{v}| \leq \omega N$) and $\widetilde{v}=\widetilde{u}$, we infer that $v \in D_{\omega N}(\Psi+J)$. Then, from \eqref{demcor} it follows
$$\mathcal{E}_h(u)=\mathcal{E}_h\left(\widetilde{u}+\sum_{i=1}^{N}\langle \overline{u} \, | \, e^i \rangle \, e^i-\sum_{i=1}^{N}k_i\omega_ie^i\right)=\mathcal{E}_h(\widetilde{v}+\overline{v})=\mathcal{E}_h(v),$$
which completes the proof.
\cqfd

Below, for $F$ satisfying $(H_{F})$, we denote
$$\overline{F}(v):=\frac{1}{T}\sum_{n=1}^{T}F(n,v) \qquad (v\in \mathbb{R}^N).$$

\begin{theorem}\label{mininfi}
If there are constants $\alpha,c\ge 0$ such that
\begin{equation}\label{rel1cor}
    |\nabla_vF(n,v)|\leq c(|v|^\alpha+1) \qquad (n\in\mathbb{Z}[1,T],\, v\in\mathbb{R}^N)
\end{equation}
and
\begin{equation}\label{rel2cor}
    \lim_{|v|\to\infty}|v|^{-\alpha}[\overline{F}(v)+\langle \overline{h}|v\rangle]=-\infty,
\end{equation}
then the functional $\mathcal{E}_h$ is bounded from below and attains its infimum at some $u\in D(\Psi + J)$, which is a solution of problem \eqref{forcterm}.
\end{theorem}

\noindent \proof
For $u\in D(\Psi+J)$ we have
\begin{eqnarray}\label{sdpt00}
    \sum_{n=1}^T F(n,u(n))&=&\sum_{n=1}^T [F(n,u(n))-F(n,\overline{u})+F(n,\overline{u})]\nonumber\\
    &=&\sum_{n=1}^T \int_{0}^{1} \frac{d}{ds}F(n,\overline{u}+s\widetilde{u}(n))ds+\sum_{n=1}^{T}F(n,\overline{u})\nonumber\\
    &=&\sum_{n=1}^{T}\int_{0}^{1}\langle \nabla_u F(n,\overline{u}+s\widetilde{u}(n))|\widetilde{u}(n)\rangle ds+\sum_{n=1}^{T}F(n,\overline{u}).
\end{eqnarray}
Using \eqref{rel1cor}, \eqref{estvtilde} and inequality
\begin{equation}\label{ineg2alpha}
    |x+y|^\alpha\leq 2^\alpha(|x|^\alpha+|y|^\alpha)\quad (x,y\in\mathbb{R}^N, \mbox{ with convention }0^0=0),
\end{equation}
we get the estimation
\begin{equation*}
\begin{split}
\sum_{n=1}^{T}F(n,u(n))&\leq c[2^\alpha(|\overline{u}|^\alpha+(T+1)^\alpha a^\alpha)+1]T(T+1)a+\sum_{n=1}^{T}F(n,\overline{u})\\
&\leq c[2^\alpha(|\overline{u}|^\alpha+(T+1)^\alpha a^\alpha)+1]T(T+1)\, a+T\overline{F}(\overline{u}).
\end{split}
\end{equation*}

\noindent This, together with
\begin{eqnarray}\label{sdpt01}
    \sum_{n=1}^{T}\langle h(n)|u(n)\rangle&=&\sum_{n=1}^{T}\langle h(n)|\overline{u}\rangle +\sum_{n=1}^{T}\langle h(n)|\widetilde{u}(n)\rangle\\
        &\leq& T\langle \overline{h}|\overline{u} \rangle+T(T+1)\, a|h|_\infty,\nonumber
\end{eqnarray}
yields
\begin{equation*}
\sum_{n=1}^{T}\left[F(n,u(n))+\langle h(n)|u(n)\rangle\right]\leq T|\overline{u}|^\alpha\left\{|\overline{u}|^{-\alpha}[\overline{F}(\overline{u})+\langle \overline{h}|\overline{u}\rangle]+c(T+1)\, a2^\alpha\right\}+c_1,
\end{equation*}
where
\begin{equation}\label{c1def}
    c_1:=\left[c\, (2^\alpha (T+1)^\alpha \, a^\alpha+1)+|h|_\infty \right]T(T+1)\, a.
\end{equation}
Thus, as $\Psi+J$ is bounded from below -- say by constant $c_2\in \mathbb{R}$, it holds
\begin{equation}\label{sdpt02}
\begin{split}
\mathcal{E}_h(u)&=(\Psi+J)(u)-\sum_{n=1}^{T}[F(n,u(n))+\langle h(n)|u(n)\rangle]\\
&\geq c_3-T|\overline{u}|^\alpha\left\{|\overline{u}|^{-\alpha}[\overline{F}(\overline{u})+\langle \overline{h}|\overline{u}\rangle]+c(T+1)a2^\alpha\right\} \qquad (u\in D(\Psi+J)),
\end{split}
\end{equation}
with $c_3:=c_2-c_1$. Then, on account of \eqref{rel2cor}, there exists some $\delta>0$ such that $\mathcal{E}_h(u)>0$ for all $u\in D(\Psi+J)$ with $|\overline{u}|>\delta$. Therefore, as $\mathcal{E}_h(0_{X_{T+2}})=0$, it follows that \eqref{relinf} is verified and the proof is concluded by  Theorem \ref{thmminimumenergy}.
\cqfd

\begin{theorem}\label{saddlepoint}
Assume that $j|_{d_N^1}\equiv0$ and $j|_{D(j)}$ is bounded. If there are constants $\alpha, c\ge 0$ such that \eqref{rel1cor} holds true and
\begin{equation}\label{sdpt1}
    \lim_{|v|\to\infty}|v|^{-\alpha}[\overline{F}(v)+\langle \overline{h}|v\rangle]=+\infty,
\end{equation}
then problem \eqref{forcterm} has at least one solution.
\end{theorem}
\noindent \proof We show that $\mathcal{E}_h$ has the geometry required by Saddle Point Theorem \cite[Theorem 3.5]{Sz}. In this view, let us split $X_{T+2}=\mathbb{R}^N\oplus\widetilde{X}_{T+2}$, where $\widetilde{X}_{T+2}=\{u\in {X}_{T+2} \, : \, \overline{u}=0_{\mathbb{R}^N} \}$.
Since for $u\equiv x\in \mathbb{R}^N$, one has $\Psi(x)=0$, $J(x)=0$ as $j|_{d_N^1}\equiv0$, from \eqref{sdpt1} it follows
\begin{equation}\label{sdpt2}
\mathcal{E}_h(x)=-\sum_{n=1}^{T}[F(n,x)+\langle h(n)|x \rangle]=-T|x|^\alpha\left\{|x|^{-\alpha}[\overline{F}(x)+\langle \overline{h}|x\rangle]\right\}
\to -\infty , \quad \mbox{ as } |x| \to \infty.
\end{equation}
On the other hand, by virtue of \eqref{sdpt02}, we infer
\begin{equation}\label{sdpt3}
\mathcal{E}_h(u)\geq c_3 \qquad ( u \in D(\Psi+J)\cap\widetilde{X}_{T+2}).
\end{equation}
Then, from \eqref{sdpt2} and \eqref{sdpt3} there are constants $r>0$ and $c_4<c_3$ so that
\begin{equation}\label{saddlecond}
\mathcal{E}_h|_{\{ x\in \mathbb{R}^N \, :  \, |x|=r \}}\leq c_4 \quad \mbox{and} \quad \mathcal{E}_h|_{\widetilde{X}_{T+2}}\geq c_3.
\end{equation}

It remains to verify that $\mathcal{E}_h$ satisfies the (PS)-condition. With this aim, let $\{u_k\}\subset D(\Psi + J)$ be a (PS)-sequence. First,
proceeding as in the proof of Theorem \ref{mininfi}, one has (see \eqref{sdpt00} and \eqref{sdpt01})
\begin{eqnarray*}
    \sum_{n=1}^T [F(n,u_k(n))+\langle h(n)|u_k(n)\rangle]-\sum_{n=1}^{T}[F(n,\overline{u}_k)+\langle h(n)|\overline{u}_k\rangle]=\\
    \sum_{n=1}^{T}\int_{0}^{1}\langle \nabla_u F(n,\overline{u}_k+s\widetilde{u}_k(n))|\widetilde{u}_k(n)\rangle ds+\sum_{n=1}^{T}\langle h(n)|\widetilde{u}_k(n)\rangle
\end{eqnarray*}
and using \eqref{rel1cor}, \eqref{estvtilde} and \eqref{ineg2alpha} we obtain
\begin{equation*}
    \left|\sum_{n=1}^T \left[F(n,u_k(n))+\langle h(n)|u_k(n)\rangle - F(n,\overline{u}_k)-\langle h(n)|\overline{u}_k\rangle\right]\right|\leq
    c_5|\overline{u}_k|^{\alpha}+c_1,
\end{equation*}
where $c_5:=c\, 2^\alpha\, T(T+1)\, a$ and $c_1$ given in \eqref{c1def}. Next, since $\{\mathcal{E}_h(u_k) \}$ and $\{(\Psi+J)(u_k)\}$ are bounded, from the writing
$$\mathcal{E}_h(u_k)=(\Psi+J)(u_k)-T[\overline{F}(\overline{u}_k)+\langle\overline{h}|\overline{u}_k\rangle]-\sum_{n=1}^T \left[F(n,u_k(n))+\langle h(n)|u_k(n)\rangle - F(n,\overline{u}_k)-\langle h(n)|\overline{u}_k\rangle\right]$$
it follows that there is a constant $c_6 \in \mathbb{R}$ such that
$$T[\overline{F}(\overline{u}_k)+\langle\overline{h}|\overline{u}_k\rangle]\leq c_6+c_5|\overline{u}_k|^{\alpha}+c_1,$$
i.e.,
$$T|\overline{u}_k|^\alpha\left\{|\overline{u}_k|^{-\alpha}[\overline{F}(\overline{u}_k)+\langle\overline{h}|\overline{u}_k\rangle]-c_5/T\right\}\leq c_1+c_6.$$
Then, by \eqref{sdpt1} the sequence $\{\overline{u}_k\}$ is bounded in $\mathbb{R}^N$. Hence, as $\{\widetilde{u}_k\}$ is bounded in the finite dimensional space $X_{T+2}$ (see \eqref{estvtilde}), one get that functional $\mathcal{E}_h$ satisfies the (PS)-condition. Thus, by \eqref{saddlecond} and Saddle Point Theorem, $\mathcal{E}_h$ has a critical point and the proof is achieved by virtue of Theorem \ref{cpsol}. \cqfd

\begin{example}\label{exampletang1}{\em For given $\nu=(\nu_1, \dots, \nu_N)\in \mathbb{R}^N$, we define
$$ SIN_{\nu}:\mathbb{R}^N\to\mathbb{R}^N, \qquad  SIN_{\nu}(v)=\left (\nu_1 \sin v_1, \dots, \nu_N \sin v_N \right ) \quad (v=(v_1, \dots , v_N) \in \mathbb{R}^N)$$
and consider the system
\begin{equation}\label{execuation}
-\Delta\left[\phi(\Delta u(n-1))\right]=b(n)|u(n)|^{\alpha-1}u(n)+c(n)SIN_{\nu}(u(n))+h(n) \quad (n \in \mathbb{Z}[1,T]),
\end{equation}
with $\alpha>0$ a constant, $b,\ c:\mathbb{Z}[1,T] \to \mathbb{R}$ and $h\in X_{T}$.
As $SIN_{\nu}(v)=\nabla C_{\nu}(v)$, where $C_{\nu}(v)=-\displaystyle\sum_{i=1}^{N}\nu_i\cos v_i$,
this system associated with the boundary condition \eqref{potbcsysgen} is a problem of type \eqref{forcterm}, with
\begin{equation*}\label{Fex}
F(n,v)=b(n) \frac{|v|^{\alpha +1}}{\alpha +1}+c(n) \, C_{\nu}(v) \quad \left ( n \in \mathbb{Z}[1,T], \, v\in \mathbb{R}^N \right ).
\end{equation*}
Then, by Theorem \ref{mininfi} it is easy to see that system \eqref{execuation} has at least one solution satisfying the general boundary condition \eqref{potbcsysgen} provided that $\overline{b}<0$. This result is known in the particular case of the periodic boundary conditions \cite[Example 10]{Ma1}. Also, from Theorem \ref{saddlepoint}, we obtain that system \eqref{execuation} subject to the boundary condition \eqref{potbcsysgen} has at least one solution provided that $j|_{d_N^1}\equiv0$, $j|_{D(j)}$ is bounded and $\overline{b}>0$.
\smallskip

If $b\equiv 0$ then \eqref{execuation} becomes
\begin{equation}\label{execuationb0}
-\Delta\left[\phi(\Delta u(n-1))\right]=c(n)SIN_{\nu}(u(n))+h(n) \quad (n \in \mathbb{Z}[1,T])
\end{equation}
and, using Theorem \ref{thperiodic} with $\omega_i=2 \pi$ ($i=\overline{1,N}$), we infer that \eqref{execuationb0} has at least one solution satisfying the boundary condition \eqref{potbcsysgen} provided that $j$ fulfills condition $(i)$ in Theorem \ref{thperiodic} and $\overline{h}=0_{\mathbb{R}^N}$.
\smallskip

Notice that $j$ which yields the periodic conditions \eqref{peridis} as well as the Neumann ones \eqref{neuhdis} satisfies both the requirement $(i)$ of Theorem \ref{thperiodic} and those of Theorem \ref{saddlepoint}. In addition to these classical conditions, generalizations of \eqref{neuhdis} can be obtained, as seen below.
}
\end{example}

Making use of the strip-like set $D_{\sigma}$ introduced in \eqref{Dsigma}, we analyze the boundary condition \eqref{ncon} with $K=\overline{D}_{\sigma}$ when $\sigma \geq (T+1)\, a$ (the remaining case $0 < \sigma<  (T+1) \, a$ can be treated similarly to \cite[Proposition 4.7]{Je}).
Then \eqref{ncon}, with $G=G(x,y)$ as in Section 1, becomes
\begin{equation}\label{nconstrip}
\left\{ \begin{array}{clrc}
& \! \! \! \! \! \left ( u(0), u(T+1) \right ) \in \overline{D}_{\sigma}, \\
& \! \! \! \! \!  \left ( \phi (\Delta u(0))-\nabla _x G ( u(0), u(T+1)), -\phi (\Delta u(T))-\nabla _y G ( u(0), u(T+1))\right ) \in N_{\overline{D}_{\sigma}}(u(0), u(T+1))
\end{array}
\right.
\end{equation}
and, recall this is nothing else but \eqref{potbcsysgen} with
\begin{equation}\label{partjS}
j=j_{G,\overline{D}_{\sigma}}:=G+I_{\overline{D}_{\sigma}}.
\end{equation}

\begin{proposition}\label{condsigma} Assume that $G=G(x,y):\mathbb{R}^N \times \mathbb{R}^N \to \mathbb{R}$ is convex, $\nabla G= \left ( \nabla_x G  \, , \, \nabla_y G \right )$ exists and is continuous,
$\nabla_x G(0_{\mathbb{R}^N \times \mathbb{R}^N}) = 0_{\mathbb{R}^N}=\nabla_y G(0_{\mathbb{R}^N \times \mathbb{R}^N})$
and  $G(0_{\mathbb{R}^N \times \mathbb{R}^N})=0$. If $\sigma \geq (T+1)\, a$ and $u$ is a solution of problem \eqref{forcterm} with $j= j_{G,\overline{D}_{\sigma}}$, then it satisfies the Neumann-Steklov type boundary conditions
\begin{equation}\label{neumsteklov}
\phi (\Delta u(0))=\nabla _x G ( u(0), u(T+1)), \quad \phi (\Delta u(T))=-\nabla _y G ( u(0), u(T+1)).
\end{equation}
\end{proposition}

\noindent\proof As $u$ is a solution of \eqref{forcterm}, from \eqref{velongs} we know that $(u(0),u(T+1)) \in  D_{(T+1)\, a} \subset D_{\sigma}$. Then, \eqref{neumsteklov} follows from \eqref{nconstrip} and
$N_{\overline{D}_{\sigma}}(z) =\{ 0_{\mathbb{R}^N \times \mathbb{R}^N} \}$ ($z\in D_{\sigma}$). \cqfd

\begin{corollary}\label{thperneuste} Let $g:\mathbb{R}^N \to \mathbb{R}$ be of class $C^1$, convex, with $\nabla g (0_{\mathbb{R}^N})=0_{\mathbb{R}^N}$, $g(0_{\mathbb{R}^N})=0$ and $h\in X_T$.
If either $F$ satisfies the periodicity condition (ii) in Theorem \ref{thperiodic} and $\overline{h}=0_{\mathbb{R}^N}$ or there are constants $\alpha, c\ge 0$ such that \eqref{rel1cor} and \eqref{sdpt1} hold true,
then problem
\begin{equation}\label{forctermpart1}
\left\{
  \begin{array}{lcl}
    -\Delta\left[\phi(\Delta u(n-1))\right]=\nabla_u F(n, u(n)) +h(n) \quad (n \in \mathbb{Z}[1,T]) \\
    \phi (\Delta u(0))=\nabla g ( u(0)- u(T+1))=\phi (\Delta u(T))
  \end{array}
  \right.
  \end{equation}
has at least one solution.
\end{corollary}

\noindent\proof Take $G(x,y):=g(x-y)$ ($x,y \in \mathbb{R}^N$) and let $\sigma \geq (T+1)\, a$. Clearly, one has that $G$ is convex, of class $C^1$ on $\mathbb{R}^N \times \mathbb{R}^N$, $G|_{d_N^1}\equiv 0$,
$\nabla G(x,y)= \left ( \nabla g(x-y), -\nabla g(x-y) \right ),$
and $\nabla G|_{d_N^1}\equiv0_{\mathbb{R}^N \times \mathbb{R}^N}.$ With this $G$, function  $j=j_{G,\overline{D}_{\sigma}}$ in \eqref{partjS} satisfies condition $(i)$ from Theorem \ref{thperiodic},
as well as $j|_{d_N^1}\equiv0$ and $j|_{D(j)}$ is bounded (since $D(j)=\overline{D}_{\sigma}$ and $g$ is bounded on the compact set $\overline{B}_{\sigma}$). Then the conclusion follows from Proposition \ref{condsigma} and Theorems \ref{thperiodic} and \ref{saddlepoint}. \cqfd

\medskip
As a simple example of a boundary condition of the type of the one in \eqref{forctermpart1} we can give
\begin{equation}\label{examplesdp}
\phi (\Delta u(0))=| u(0)- u(T+1)| ( u(0)- u(T+1))=\phi (\Delta u(T)),
\end{equation}
which corresponds to $g(x)=|x|^2/2$ ($x\in \mathbb{R}^N$). Then, referring to Example \ref{exampletang1}, from Corollary \ref{thperneuste} we have that system \eqref{execuation} has at least one solution $u$ satisfying  \eqref{examplesdp} provided that $\overline{b}>0$, and the same is true for system \eqref{execuationb0} when $\overline{h}=0_{\mathbb{R}^N}$.

\medskip
\begin{remark}\label{remuges}
{\em Theorem \ref{mininfi} substantially extends to the general case of the boundary condition \eqref{potbcsysgen} the result from the particular case of periodic conditions due to Mawhin \cite[Theorem 7]{Ma1}. On the other hand, in \cite[Section 8]{Ma1} the existence of periodic solutions when the potential satisfies condition \eqref{sdpt1} together with \eqref{rel1cor} and  which could lead to saddle point solutions is raised as an open problem. In this sense, Theorem \ref{saddlepoint} not only answers this problem, but also highlights a whole class of problems for which the result is valid. The averaged coercivity conditions \eqref{rel2cor} and \eqref{sdpt1} are of the type introduced  by Tang  for classical second order differential systems \cite{Tg} and by  Xue and Tang for difference systems \cite{XuTg}. In the particular case $\alpha=0$ these conditions are of Ahmad-Lazer-Paul
type \cite{ALP} and are used for discrete periodic and Neumann problems with $p(\cdot)$-Laplacian in \cite[Theorems 3.5 and 6.3]{BeJeSe}. It is worth noting that, based on the provided variational formulation, various other results on the existence of solutions can be obtained by following the strategies driven from the differential case \cite{Je}. However, for the sake of brevity, here we limit ourselves to those presented above.
}
\end{remark}

\Addresses


\bigskip




\begin{thebibliography}{99}

\bibitem{ALP} S. Ahmad, A.C. Lazer and J.L. Paul, Elementary critical point theory and perturbations of elliptic boundary value problems at resonance, \textit{Indiana Math. J.} \textbf{25} (1976), 933--944.

\bibitem{BeGh} C. Bereanu and D. Gheorghe, Topological methods for boundary value problems involving discrete vector $\phi$-Laplacians, \textit{Topol. Methods Nonlinear Anal.} \textbf{38} (2) (2011), 265--276.

\bibitem{BeJeSe1} C. Bereanu, P. Jebelean and C. \c{S}erban, Ground state and mountain pass solutions for discrete $p(\cdot)$-Laplacian, \textit{Bound. Value Probl.} \textbf{2012}, 2012:104, 1--13.

\bibitem{BeJeSe} C. Bereanu, P. Jebelean and C. \c{S}erban, Periodic and Neumann problems for discrete $p(\cdot)$-Laplacian, \textit{J. Math. Anal. Appl.} \textbf{399} (1) (2013), 75--87.

\bibitem{BeMa} C. Bereanu and J. Mawhin, Boundary value problems for second-order nonlinear difference equations with discrete
$\phi$-Laplacian and singular $\phi$, \textit{J. Difference Equ. Appl.} \textbf{14} (2008), 1099--1118.

\bibitem{berg} P.G. Bergmann, \textit{Introduction to the Theory of Relativity}, Dover, New York, 1976.

\bibitem{Br} H. Brezis, \textit{Op\'{e}rateurs Maximaux Monotones}, North-Holland, Amsterdam, 1973.

\bibitem{CaJeSe} A. Cabada, P. Jebelean and C. \c{S}erban, Dirichlet systems with discrete relativistic operator, \textit{Bull. Lond. Math. Soc.} \textbf{56} (3) (2024), 1149--1168.

\bibitem{DaGu} A. Daouas and A. Guefrej, On subharmonic solutions for second order difference equations with relativistic operator, \textit{Bol. Soc. Paran. Mat.} \textbf{43} (2025), 1--11.

\bibitem{DMM} Z. Do\v{s}l\'{a}, M. Marini and S. Matucci,  Extremal solutions for difference equations with the Minkowski mean curvature operator, \textit{J. Difference Equ. Appl.}, DOI: 10.1080/10236198.2026.2637719.

\bibitem{GrSe} A. Gruie and C. \c{S}erban, Non-homogeneous discrete Dirichlet problem with singular $\phi$-Laplacian, \textit{Appl. Math. Lett.} \textbf{176} (2026), 109859.

\bibitem{Je} P. Jebelean, Potential systems with singular $\phi$-Laplacian, \textit{Commun. Contemp. Math.} {\bf 28} (3) (2026), 2550041.

\bibitem{JeMaSe} P. Jebelean, J. Mawhin and C. \c{S}erban, A vector $p$-Laplacian type approach to multiple periodic solution for the $p$-relativistic operator, \textit{Commun. Contemp. Math.} {\bf 19} (2017), 1--16.

\bibitem{JePr} P. Jebelean and R. Precup, Poincar\'{e} inequality in reflexive cones, \textit{Appl. Math. Lett.} {\bf 24} (2011), 359--363.

\bibitem{JeSe} P. Jebelean and C. \c{S}erban, Non-potential systems with relativistic operators and maximal monotone boundary conditions, \textit{J. Fixed Point Theory Appl.} \textbf{27} (2) (2025), 1--11.

\bibitem{LaLi} L.D. Landau and E.M. Lifchitz, \textit{The Classical Theory of Fields}, Pergamon Press, Oxford, 1971.

\bibitem{LiLu} Y. Li and Y. Lu, Ambrosetti-Prodi type results for discrete Minkowski - mean curvature operators with repulsive singularities, \textit{J. Appl. Anal. Comput.} \textbf{15} (5) (2025), 2726--2746.

\bibitem{LuMa} Y.Q. Lu and R.Y. Ma, Existence and multiplicity of solutions of second-order discrete Neumann problem with singular $\phi$-Laplacian operator, \textit{Adv. Difference Equ.} \textbf{2014}, 2014:227, 1--18.

\bibitem{LuMaLu} Y.Q. Lu, R.Y. Ma and B. Lu, Existence and multiplicity of solutions for discrete Neumann-Steklov problems with singular $\phi$-Laplacian, \textit{Adv. Difference Equ.} \textbf{2017},  2017:312, 1--18.

\bibitem{Llo} N.G. Lloyd, \textit{Degree Theory}, Cambridge University Press, Cambridge, 1978.

\bibitem{Ma1} J. Mawhin, Periodic solutions of second order nonlinear difference systems with $\phi$-Laplacian: A variational approach, \textit{Nonlinear Anal.} \textbf{75} (2012), 4672--4687.

\bibitem{[Ph]} R.R. Phelps, Lectures on maximal monotone operators, \textit{Extracta Math.} {\bf 12} (1997), 193--230.

\bibitem{Rock2} R.T. Rockafellar, On the maximal monotonicity of subdifferential mappings, \textit{Pacific J. Math.} {\bf 33} (1970), 209--216.

\bibitem{Rock} R.T. Rockafellar, \textit{Convex Analysis}, Princeton University Press, Princeton, 1972.

\bibitem {Sz} A. Szulkin, Minimax principles for lower semicontinuous functions and applications to nonlinear boundary value problems, \textit{Ann. Inst. H. Poincar\'{e} Anal. Non Lin\'{e}aire} {\bf 3} (1986), 77--109.

\bibitem{cSe} C. \c{S}erban, Existence of solutions for discrete $p$-Laplacian with potential boundary conditions, \textit{J. Difference Equ. Appl.} \textbf{19} (3) (2013), 527--537.

\bibitem{Tg} C.-L. Tang,  Periodic solutions for nonautonomous second order systems with sublinear nonlinearity, \textit{Proc. Amer. Math. Soc.} \textbf{126} (1998), 3263--3270.

\bibitem{Ze} E. Zeidler, \textit{Nonlinear Functional Analysis and its Applications II/B: Nonlinear Monotone Operators}, Springer-Verlag New York, 1990.

\bibitem{XuTg} Y.F. Xue and C.-L. Tang, Existence of a periodic solution for subquadratic second-order discrete Hamiltonian system, \textit{Nonlinear Anal.} \textbf{67} (2007), 2072--2080.

\end{thebibliography}
\end{document}